\documentclass[11pt,a4paper]{article}
\usepackage[ngerman,british]{babel}
\usepackage[utf8]{inputenc}
\usepackage{amsmath}
\usepackage{cancel}
\usepackage[normalem]{ulem}
\usepackage{amsthm}
\usepackage{amsfonts}
\usepackage{amssymb}
\usepackage{dsfont}
\usepackage[T1]{fontenc}
\usepackage{makeidx}
\usepackage{units}
\usepackage[pdftex]{graphicx} 
\usepackage[dvipsnames]{xcolor}
\usepackage{graphicx}
\usepackage{floatflt} 
\usepackage{verbatim}
\usepackage{wrapfig} 
\usepackage{caption}
\usepackage{latexsym}
\usepackage{setspace}
\usepackage{hyperref}
\usepackage{enumitem}

\allowdisplaybreaks[4]

\newtheorem{theorem}{Theorem}[section]
\newtheorem{prop}[theorem]{Proposition}
\newtheorem{lem}[theorem]{Lemma}
\newtheorem{cor}[theorem]{Corollary}
\theoremstyle{definition}

\newtheorem{ass}[theorem]{Assumption}
\newtheorem{remark}[theorem]{Remark}

\newtheorem{conv}[theorem]{Convention}

\theoremstyle{plain}

\newcommand{\NN}{\mathbb{N}}

\newcommand{\RR}{\mathbb{R}}
\newcommand{\PP}{\mathbb{P}}
\newcommand{\CC}{\mathbb{C}}

\newcommand{\FF}{\mathbb{F}}
\newcommand{\Fbb}{\mathbb{F}}

\newcommand{\EE}{\mathbb{E}}

\newcommand{\cS}{\mathcal{S}}

\newcommand{\cE}{\mathcal{E}}
\newcommand{\CE}{\mathcal{E}}
\newcommand{\CF}{\mathcal{F}}

\newcommand{\cov}{\mathbb{C}\mathrm{ov}}
\newcommand{\var}{\mathbb{V}\mathrm{ar}}

\newcommand{\rmd}{{\rm d}}

\newcommand{\ee}{\mathrm{e}}

\newcommand{\taure}{\tau^{re}}
\newcommand{\tauex}{\tau^{ex}}

\newcommand{\fatpi}{\boldsymbol{\pi}}
\newcommand{\fateps}{\boldsymbol{\epsilon}}
\newcommand{\fatone}{\mathbf{1}}

\newcommand{\ol}{\overline}
\newcommand{\wt}{\widetilde}

\newcommand{\fat}[1]{\mathbf{#1}}
\newcommand{\fatQ}{\mathbf{Q}}
\newcommand{\fatLambda}{\mathbf{\Lambda}}

\newcommand{\Lm}{\Lambda}

\newcommand{\Ebb}{\mathbb{E}}
\newcommand{\Pbb}{\mathbb{P}}
\newcommand{\Rbb}{\mathbb{R}}

\newcommand{\p}{^}
\newcommand{\hatfat}[1]{\fat{\hat{#1}}}

\DeclareMathOperator{\diag}{\mathbf{diag}}

\DeclareMathOperator*{\iv}{\overset{d}{=}}

\DeclareMathOperator{\Var}{\mathbb{V}ar}

\newcommand\normal{\color{black}}
\newcommand{\pa}{\color{black}}
\newcommand\anita{\color{black}}

\newcommand{\om}{\omega}

\newcommand{\sig}{\sigma}

\newcommand{\Bscr}{\mathcal B}
\newcommand{\Fscr}{\mathcal F}
\newcommand{\Pscr}{\mathcal P}

\begin{document}
	\title{On moments of integrals with respect to Markov additive processes and of Markov modulated generalized Ornstein-Uhlenbeck processes}
	\author{Anita Behme\thanks{Technische Universit\"at
			Dresden, Institut f\"ur Mathematische Stochastik, Helmholtzstra{\ss}e 10, 01069 Dresden, Germany, \texttt{anita.behme@tu-dresden.de} and \texttt{paolo.di\_tella@tu-dresden.de}, phone: +49-351-463-32425, fax:  +49-351-463-37251.}\;, Paolo Di Tella$^\ast$ and Apostolos Sideris$^\ast$}
	\date{\today}
	\maketitle
	
	\vspace{-1cm}
	\begin{abstract}
	We establish sufficient conditions for the existence, and derive explicit formulas for the $\kappa$'th moments, $\kappa \geq 1$, of Markov modulated generalized Ornstein-Uhlenbeck processes as well as their stationary distributions. In particular, the running mean, the autocovariance function, and integer moments of the stationary distribution are derived in terms of the characteristics of the driving Markov additive process. \\
	Our derivations rely on new general results on moments of Markov additive processes and (multidimensional) integrals with respect to Markov additive processes. 
	\end{abstract}

	

\section{Introduction} \label{S0} \setcounter{equation}{0}

Given a bivariate Markov additive process (MAP) $((\xi,\eta),J)=((\xi_t, \eta_t),J_t)_{t\geq 0}$, the \emph{Markov modulated generalized Ornstein-Uhlenbeck  (MMGOU) process driven by $((\xi,\eta),J)$} has been defined in \cite{BEHME+SIDERIS_MMGOU} as the process $(V_t)_{t\geq 0}$ given by 
\begin{align}\label{MMGOUexplicit}
	V_t=e^{-\xi_t} \bigg(V_0+\int_{(0,t]} \ee^{\xi_{s-}} \rmd\eta_s\bigg), \quad t\geq 0,
\end{align}
where the starting random variable $V_0$ is conditionally independent of $((\xi_t, \eta_t),J_t)_{t\geq 0}$ given $J_0$.
{\pa As shown in \cite[Prop.\ 2.7]{BEHME+SIDERIS_MMGOU}} the MMGOU process is the unique solution of the stochastic differential equation 
\begin{align}\label{MMGOUSDE}
	\rmd V_t=V_{t-}\rmd U_t+\rmd L_t,\qquad t\geq 0, 
\end{align}
for another bivariate MAP $((U,L),J)=((U_t, L_t),J_t)_{t\geq 0}$ with $\Delta U> -1$ which is uniquely determined by $((\xi,\eta),J)$. {\pa Hereby, for any c\`adl\`ag process $Z$ we denote by $Z_{t-}$ the left-hand limit of $Z$ at time $t\in(0,\infty)$, set $Z_{0-}:=Z_0$, and write $\Delta Z_t=Z_t-Z_{t-}$ for its jumps}.

Assuming that the driving Markov chain $J$ is defined on a finite state space $S$ and is ergodic with stationary distribution $\pi$, it has been shown in \cite[Thm. 3.3(a) \& Rem. 3.4.2]{BEHME+SIDERIS_MMGOU} that the MMGOU process admits a \emph{non-trivial stationary distribution} if and only if the integral $\int_{(0,t]} \ee^{\xi^\ast_{s-}} \rmd L^*_s$ converges {\pa $\PP_\pi^*$-almost} surely as $t\rightarrow \infty$ to some finite-valued random variable $V_\infty$. Hereby, $((\xi^\ast, L^\ast), J^\ast)$ denotes the time-reversal of the Markov additive process $((\xi,L),J)$ . {\pa In this case,} the stationary distribution of the MMGOU process under $\PP_\pi$ is uniquely determined as the distribution of 
\begin{equation} \label{eq_defVinfty} V_\infty= - \int_{(0,\infty)} \ee^{\xi^\ast_{s-}} \rmd L^*_s. \end{equation}
Necessary and sufficient conditions for \pa$\PP_\pi^\ast$-a.s. \normal convergence of integrals of the form \eqref{eq_defVinfty} in terms of the characteristics of the appearing processes have been given in \cite{BEHME+SIDERIS_ExpFuncMAP2020}. \\
\pa Besides, by \cite[Thm. 3.3(b)]{BEHME+SIDERIS_MMGOU}, the MMGOU process can admit \emph{trivial stationary distributions} whenever $V$ degenerates to a continuous-time Markov chain that is piecewise constant with discrete stationary distribution. This behaviour will typically be excluded in this paper. \normal

In the special case that $(\xi,\eta)$ (equivalently $(U,L)$) is a bivariate Lévy process, the MMGOU process reduces to the \emph{generalized Ornstein-Uhlenbeck (GOU) process}, originally introduced in \cite{DEHAAN+KARANDIKAR_EmbeddingSDEcontinuoustimeprocess} and subsequently studied and applied by various authors, such as  \cite{BNS01, BEHME+LINDNER+MALLER_StationarySolutionsSDEGOU, Kevei, KLM:2004, LINDNER+MALLER_LevyintegralsStationarityGOUP,   MALLER+MUELLER+SZIMAYER_OUPExtensions, PAULSEN_RisksStochasticEnvironment1992}, to name just a few. Second-order properties of GOU processes like moments and tail distributions are specifically studied in \cite{BEHME_DistributionalPropertiesGOULevynoise}, as well as in \cite{BNS01, KLM:2004, LINDNER+MALLER_LevyintegralsStationarityGOUP} for special cases.\\
Another special case of the MMGOU process that has gained attention in the last few years is the \emph{Markov modulated Ornstein-Uhlenbeck (MMOU) process}. This can be obtained by specifying that \pa $U_t= -\gamma_{J_t} t$ is a piecewise linear MAP, while $L_t=\sigma_{J_t} B_t + \alpha_{J_t} t$ is a Markov modulated Brownian motion with drift, where $(\gamma_j,\alpha_j,\sigma_j)$ are constants with $\sigma_j>0, \gamma_j\geq 0$, and $\gamma_j>0$ for at least one $j\in S$. The MMOU process has been introduced and studied first in \cite{HUANG+SPREJ_MarkovmodulatedOUP2014}, where the authors also derive the mean and (autoco-)variance function of the MMOU process using methods from stochastic analysis, while a recursion for higher moments is derived using Laplace transforms.  Other articles like e.g.  \cite{LINDSKOG+MAJUMDER_ExactLongtimebehaviourMAP2019} and \cite{ZHANG+WANG_StationarydistributionOUP2stateMarkovswitching2017} focused on stationarity properties of MMOU processes. More precisely, in \cite{LINDSKOG+MAJUMDER_ExactLongtimebehaviourMAP2019}, assuming $\alpha_j\equiv 0$, conditions for stationarity of MMOU processes and some expressions for the stationary distribution as scale mixtures are derived. In \cite{ZHANG+WANG_StationarydistributionOUP2stateMarkovswitching2017}, again assuming $\alpha_j\equiv 0$, expressions for the Fourier transform of the density of an MMOU process with two-state Markov switching (i.e. $|S|=2$) are provided. \normal \\
As the stationary distribution of the MMGOU process is the distribution of an exponential functional, at this point we also mention \cite{SALMINENVOSTRIKOVA}, where moments of exponential functionals with deterministic integrator and with a general process with independent increments in the exponential are studied.  
 
In this paper we derive new existence results and formulas for the running mean and the autocovariance function of the MMGOU process, \pa thus generalizing the results for the MMOU process given in \cite{HUANG+SPREJ_MarkovmodulatedOUP2014}. \normal Under suitable conditions we prove that the autocovariance function is exponentially decreasing. This coincides with well-known analogue results for the (G)OU process as well as for the related discrete-time autoregressive AR(1) time series. 
Further, we derive existence results and formulas for integer moments of the stationary distribution of the MMGOU process. Here the first two moments will be {\pa stated explicitly}, while we provide a recursion formula for higher order moments. All formulas are given in terms of the characteristics of the driving MAP.\\
\anita 
In order to prove these results, after recalling various preliminaries on MAPs in Section \ref{S1}, in Section \ref{S3a} we collect existence results for (exponential) moments of MAPs. \pa Additionally, in this section we prove a general existence result for moments of stochastic integrals with respect to MAPs (Theorem \ref{Theo_mean_MAP_integral}). This theorem relies on the one hand on a new result on the existence of moments of Lévy driven stochastic integrals up to a stopping time, and on the other hand on a study of the random measure representation of the MAP's jumps at times of regime switches. \\ \anita
Once the existence of moments is established, in Section \ref{S3b} we derive explicit formulas for the mean and variance of the additive component of a MAP. Our method chosen here relies on stochastic analysis, thus avoiding the standard approach via the matrix exponential as chosen e.g. in \cite{ASMUSSEN_AppliedProbandQueues}. \pa This allows for easily interpretable formulas in terms of the characteristics of the MAP. \\ \anita
 In Section \ref{S4} the above described results on moments of the MMGOU process are presented and proven.  
 The final Section \ref{Sproofs} collects various technical proofs of results presented in Section \ref{S3}.
 \normal

\section{Markov additive processes} \label{S1} \setcounter{equation}{0}

Throughout, we assume $S=\{1,2,\ldots, |S|\}$ to be a finite set and denote its power set by $\cS$. Let  $(X,J)=(X_t,J_t)_{t\geq 0}$ be a Markov process on $\RR^d\times S$, $d\geq 1$, defined on the filtered probability space $(\Omega,\mathcal{F},\mathbb{F},\PP)$,  where $\mathbb{F}=\left(\mathcal{F}_t\right)_{t\geq 0}$ denotes a filtration satisfying the usual conditions \pa of completeness and right-continuity (see e.g. \cite{PROTTER_StochIntandSDE}), and \normal  such that $(X,J)$ is adapted. 
For any $j\in S$ we write $\PP_j(\cdot):=\PP(\cdot|J_0=j)$ and $\EE_j[\cdot]$ for the expectation with respect to $\PP_j${\pa, where we always assume $\PP(J_0=j)>0$, for every $j\in S$}. 

The Markov process $(X,J)$ is called a \emph{($d$-dimensional) Markov additive process with respect to $\FF$ ( $\FF$-MAP)}, if for all $s,t\geq 0$ and for all bounded and measurable functions $f:\RR^d\to\RR$, $g:S\to \RR$
\begin{align}\label{MAPdefinition}
{\pa	\EE\left[f(X_{s+t}-X_s)g(J_{s+t})|\mathcal{F}_s\right]=\EE_{J_s}\left[f(X_t-X_0)g(J_t)\right].}
\end{align}

Given a MAP $(X,J)$ the marginal process $X$ is called the \emph{additive component}. The process $J$, which itself is a Markov process with respect to $\Fbb$, is typically called \emph{Markovian component}, \emph{driving Markov chain/process} or \emph{modulator}. \\
\pa For any MAP, we assume throughout that  $X_0=0$.\normal 

We refer to \cite{CINLAR_MAP11972, CINLAR_MAP21972} for the first extensive studies on MAPs, to \cite{ASMUSSEN_AppliedProbandQueues} for a short textbook treatment of MAPs, and to the appendix of  \cite{DEREICH+DOERING+KYPRIANOU_RealselfsimilarprocessesStartedfromOrigin2017} for numerous results on MAPs. Basic properties of MAPs that will be needed in this article will also be recalled in the upcoming subsections.

\subsection{The additive component} 
From the definition of the MAP $(X,J)$ one can show, cf. \cite{ASMUSSEN_AppliedProbandQueues}, that there exists a sequence of independent $\RR^d$-valued Lévy processes $\lbrace X^{(j)},j\in S\rbrace$ with respect to $\FF$, each with characteristic triplet $\{(\gamma_{X^{(j)}}, \Sigma^2_{X^{(j)}}, \nu_{X^{(j)}}), j\in S \}$ {\pa with respect to the standard truncation function $\mathds{1}_{\{|x|\leq1\}}$}, and such that, whenever $J_t=j$ on some time interval $(t_1,t_2)$, the additive component $(X_t)_{t_1<t<t_2}$ behaves in law as $X^{(j)}$.\\
In this paper, we will often  consider MAPs $(X,J)$ with a bivariate additive component $X=(\zeta,\chi)$. Thus the corresponding L\'evy processes $(\zeta^{(j)},\chi^{(j)})$ are two-dimensional  and their characteristics $(\gamma_{X^{(j)}},\Sigma^2_{X^{(j)}},\nu_{X^{(j)}})$ consist of a non-random vector $\gamma_{X^{(j)}}=(\gamma_{\zeta^{(j)}},\gamma_{\chi^{(j)}})=(\gamma_{\zeta}(j),\gamma_{\chi}(j))\in\mathbb{R}^2$, a symmetric, non-negative definite $2\times 2$ matrix 
\begin{align*}
	\Sigma^2_{X^{(j)}}=\Sigma^2_X(j)=:\begin{pmatrix}
		\sigma^2_{\zeta^{(j)}}&\sigma_{\zeta^{(j)},\chi^{(j)}}\\
		\sigma_{\zeta^{(j)},\chi^{(j)}}&\sigma^2_{\chi^{(j)}}
	\end{pmatrix}=\begin{pmatrix}
		\sigma^2_{\zeta}(j)&\sigma_{\zeta,\chi}(j)\\
		\sigma_{\zeta,\chi}(j)&\sigma^2_{\chi}(j)
	\end{pmatrix}
\end{align*}
and a Lévy measure $\nu_{X^{(j)}}$ on $\mathbb{R}^2\setminus\lbrace(0,0)\rbrace$.

	Moreover, whenever the driving chain $J$ jumps {\pa for the $n$-th time at a time $T_n$, say} from state $i$ to state $j$, it induces an additional jump $Z^{ij}_{X,n}$ for $X$, whose distribution $F_{X}^{ij}$ depends only on $(i,j)$ and neither on the jump time nor on the jump number, and it is independent of any other occurring random elements. As the L\'evy processes $\{X^{(j)},j\in S\}$ have c\`adl\`ag paths, this implies that $(X,J)$ always admits a c\`adl\`ag modification and we will therefore assume any MAP $(X,J)$ to be c\`adl\`ag from now on. Moreover the above yields the following path decomposition of the additive component	\pa 
	\begin{align}\label{MAPpathdescription}
	 X_t=X_{1,t} +X_{2,t}:= \int_{(0,t]}   \rmd X_{s}^{(J_{s})}+\sum_{n\geq 1}\sum_{\substack{i,j\in S,\\i\neq j}}Z_{X,n}^{ij} \mathds{1}_{\lbrace J_{T_{n-1}}=i,J_{T_{n}}=j,T_n\leq t\rbrace}, \quad t\geq 0,
	\end{align}
	where $X_1=(X_{1,t})_{t\geq 0}$ describes the Lévy process type behaviour of $X$, while $X_2=(X_{2,t})_{t\geq 0}$ encodes the \emph{additional jumps}. \normal 
 Conversely, every process $(X,J)$ such that $(J_t)_{t\geq 0}$ is a continuous-time Markov chain with state space $S$ and $X$ has a representation as in \eqref{MAPpathdescription}, is a MAP.  
 
 We notice that the process $X_1$ in \eqref{MAPpathdescription} is a semimartingale whose characteristics depend on $J$, cf. \cite{grigelionis}. This holds also for $X_2$, {\pa it being an} adapted process of finite variation. Therefore, $X$ is a semimartingale, too, and we can define stochastic integrals with respect to $X$. 

{\pa Furthermore, according to \cite[Prop. 2 and Eq.~(41)]{DEREICH+DOERING+KYPRIANOU_RealselfsimilarprocessesStartedfromOrigin2017}, we have 
\begin{equation}\label{eq:dist.X.lev}
X_t \overset{d}= \sum_{j\in S}X\p{(j)}_{v_j(t)}+\sum_{\substack{i,j\in S\\i\neq j}}\sum_{\ell=1}\p{N_t^{i\to j}}Z\p{ij}_{X,\ell},\quad t\geq0,
\end{equation}
where $v_j(t)$ denotes the time $J$ spends in state $j$ up to time $t$, and $N_t^{i\to j}$ is the number of jumps of $J$ from $i$ to $j$ over the time interval $[0,t]$. Hereby, $Y\overset{d}= Z$ reads \emph{$Y$ and $Z$ have the same distribution}.  
}

 \subsection{The Markovian component}\label{S22}
 
  We assume throughout, that the Markovian component $J$ is irreducible, hence ergodic. We denote the intensity matrix of $J$ by  $\fatQ=(q_{ij})_{i,j\in S}$, and recall that its entries satisfy $q_{ij}>0$ for all $i,j\in S$, $i\neq j$, and $q_{ii}=-\sum_{j\in S\setminus\lbrace i\rbrace}q_{ij}$. As $J$ is assumed to be ergodic, it admits a unique stationary distribution that we denote by $\fatpi=(\pi_j)_{j\in S}$ and we write   
 \begin{align}\label{def_stationarydistirbution_Markov_}
 	\PP_\pi(\cdot)=\sum_{j\in S}\pi_j\PP_j(\cdot).
 \end{align}
 For a fixed state $j\in S$ we define the \emph{return times} to $j$ iteratively by 
$$\taure_1(j):=\inf\lbrace t>0:J_{t-}\neq j, J_t=j\rbrace, \quad \taure_n(j):=\inf\lbrace t>\taure_{n-1}(j):J_{t-}\neq j,J_t=j\rbrace,\; {\pa n\geq 2}.$$
 In a similar way, we define the  \emph{exit times} from the state $j$ by $$\tauex_1(j):=\inf\lbrace t>0: J_{t-}=j, J_t\neq j\rbrace, \quad \tauex_n(j):=\inf\lbrace t>\tauex_{n-1}(j): J_{t-}=j,J_t\neq  j\rbrace,\; {\pa n\geq 2}.$$
 Clearly, return times and exit times are stopping times with respect to $\FF$. The sequences 
 $(\taure_{n+1}(j)-\taure_n(j))_{n\geq 1}$ and $(\tauex_{n+1}(j)-\tauex_n(j))_{n\geq 1}$ are i.i.d. under any $\PP_i$. Moreover, under $\PP_j$, we may set $\taure_0(j)=0$, which yields an {\pa i.i.d.} sequence $(\taure_{n+1}(j)-\taure_n(j))_{n\geq 0}$.

Throughout, let $(N_t)_{t\geq 0}$ be the counting process that describes the number of jumps of $J$ up to time $t$. Further, let $\{T_n, n\geq 1\}$ be the sequence of jump times of $J$ and set $T_0=0$, such that in particular $N_t=\sum_{n=1}^\infty \mathds{1}_{\{T_n\leq t\}}$. 
We also define the counting processes
\begin{equation}\label{eq:def.cou.re.ex}
N^{re}_t(j):=\sum_{\ell=1}^\infty \mathds{1}_{\{\taure_\ell(j)\leq t\}},\quad  N^{ex}_t(j):=\sum_{\ell=1}^\infty \mathds{1}_{\{\tauex_\ell(j)\leq t\}}, \quad j\in S,\ t\geq0.
\end{equation}
The ergodicity of $J$ and the finiteness of the state space $S$ imply that $J$ is  positive recurrent, that is $\EE_j[\taure_1(j)]<\infty$ for all $j\in S$, cf. \cite[Chapter 3.5]{Norris}.  In our setting the return times have finite moments of all orders, i.e. $\EE_j[\taure_1(j)^k]<\infty$ for all $k\in\NN$. As we were unable to find a suitable reference for this fact, we state it here as {\pa lemma} and provide a short proof. 
 
 \begin{lem} \label{lem-finitemomentsreturn}
  	Let $J$ be an ergodic Markov chain on $S$ with $|S|<\infty$. Then  $\EE_j[(\taure_1(j))^k]<\infty$ for all $j\in S$, $k\in \NN$.
 \end{lem}
 \begin{proof}
 	Without loss of generality, we consider $j=1$. Then, under $\PP_1$, 	$$\taure_1(1) = \tauex_1(1) + \tau,$$
 	where $\tau$ is independent of $\tauex_1(1)$ and phase-type distributed. More precisely, $\tau\sim \operatorname{PH}_{|S|-1}(\mathbf{\alpha}, \mathbf{T})$ with 
 	$$\mathbf{\alpha}:=\left(\frac{q_{12}}{|q_{11}|}, \ldots, \frac{q_{1|S|}}{|q_{11}|}\right)^\top, \quad \text{and } \mathbf{T}:= (q_{ij})_{i,j=2,\ldots,|S|}.$$
 	As under $\PP_1$ the variable $\tauex_1(1)\sim \operatorname{Exp}(|q_{11}|)$ has finite moments of all orders and the same holds true for $\tau$ (see e.g. \cite[Cor. 3.1.18]{BladtNielsen}), this immediately implies the statement by an application of the binomial theorem. 	
 	\end{proof}
 
  We introduce the  \emph{localization process} $\fatLambda=(\fatLambda_{t})_{t\geq 0}$ corresponding to the Markovian component $J$ of the MAP $(X,J)$ given by 
 $$\fatLambda_t:= \fat{e}_{J_t}=(\mathds{1}_{\lbrace J_t=j\rbrace})_{j\in S}, \quad t\geq 0,$$ 
 where $\fat{e}_j$ denotes the {\pa $j$-th} unit vector, such that
 \begin{equation} \label{eq-expfatLambda} \EE_j[\fatLambda_t]= \ee^{\fatQ^\top t} \fat{e}_j.\end{equation}
 For any real-valued process $Y$ we write  $\fat{\hat{Y}}_t:=Y_t\fatLambda_t$ and note that this implies $\fatone^\top\fat{\hat{Y}}_t=Y_t$, with $\fat{1}$ denoting the column vector of $1$'s.  As $\fatLambda$ is a regular jump Markov process, by {\pa(see  \cite[Appendix B, Lemma 1.1]{ELLIOG+AGGOUN+MOORE_HiddenMarkovModelsEstimationControl1995})} there exists an $|S|$-dimensional square integrable {\pa $\FF$-martingale} $\fat M$  with respect to $\PP$, such that
 \begin{align}\label{eq-lambdadecompose}
 	\fatLambda_t=\fatQ^\top\int_{(0,t]} \fatLambda_s\rmd s+\fat{M}_t,
 \end{align}
 where 
 we make the following important notational convention for multivariate stochastic integrals:
 \begin{conv}\label{conv:integrals}
 	Throughout, the integral $\int_{(0,\cdot]} H_s\rmd Y_s$ has to be understood componentwise in the following cases: 
 	
 	(i) If $Y$ is a real-valued semimartingale and $H$ an $\RR\p d$-valued process such that its components $H\p i$, $i=1,\ldots, d$, are locally bounded adapted (if $Y$ is of finite variation) or predictable processes.
 	
 	(ii) If $Y$ is an  $\RR\p d$-valued semimartingale and $H$ a real-valued locally bounded predictable process.
 \end{conv}

 For the integral in \eqref{eq-lambdadecompose} this implies 
 \begin{align*}
 	\int_{(0,t]} \fatLambda_s \rmd s=\left(\int_{(0,t]} \mathds{1}_{\lbrace J_s=j\rbrace}\rmd s\right)_{j\in S}.
 \end{align*}
Moreover, again according to Convention \ref{conv:integrals}, integrating (componentwise) by parts, we obtain from \eqref{eq-lambdadecompose} for any real-valued semimartingale $Y$
\begin{align}\label{integrationbyparts}
	\fat{\hat{Y}}_t=Y_t\fatLambda_t=	\fat{\hat{Y}}_0+\fatQ^\top\int_{(0,t]}\fat{\hat{Y}}_{s-}\rmd s+\int_{(0,t]} Y_{s-}\rmd\fat{M}_s+\int_{(0,t]} \fatLambda_{s-}\rmd Y_s+[Y,\fatLambda]_t,
\end{align}
where $[Y,\fatLambda]:=([Y,\mathds{1}_{\{J=j\}}])_{j\in S}$.

We notice that, from \eqref{eq-lambdadecompose}$, \fat{M}$ is a martingale which is bounded on compact time intervals, since, for every $t>0$, we have $\|\fat\Lm_t\|_\infty:=\sup_{j\in |S|}|\fatLambda_t \fat{e}_j^\top|=1$ and $ \|\fatQ\|_\infty:=\sup_{i,j\in S}|q_{ij}|<\infty$. In particular,  by \cite[Eq.\ (14)]{MEYER_Inegalites}, for every predictable $H$ such that $\sup_{t\geq0}|H_t|\in L\p 1(\Pbb)$, we have $\Ebb[\sup_{0\leq s\leq t}|\int_{(0,s]} H_u\rmd\fat{M}_u|]<\infty$ for all $t\geq 0$. Therefore, we have shown the {\pa following.}
\begin{lem} \label{eq:unif.mart.int.M}
	For every predictable process $H$ with $\sup_{t\geq 0}|H_t|\in L\p 1(\PP)$ the integral process $(\int_{(0,t]} H_s\rmd \fat{M}_s)_{t\geq 0}$ is a {\pa centered martingale}.
	\end{lem} 
 
{\pa Clearly, $\fat M$ is a martingale also with respect to  $\PP_j$, for any $j\in S$, and hence with respect to $\PP_\pi$}.

\subsection{The matrix exponent}

The \emph{matrix exponent} $\fat{\Psi}_X$ of a MAP $(X,J)$ {\pa with univariate additive component $X$} is the matrix in $\CC^{|S|\times |S|}$ 
defined as
\begin{equation}\label{eq:def.psi.lap.tran}
	\fat{\Psi}_X(w)=\diag\left(\psi_j(w)\right)+\fatQ^\top\circ \left(\EE\left[\ee^{wZ^{jk}_{X,1}}\right]\right)_{j,k\in S}^\top,  
\end{equation}
for all $w\in\CC$ such that the right hand side exists, cf. \cite[Prop. XI.2.2]{ASMUSSEN_AppliedProbandQueues} or \cite{DEREICH+DOERING+KYPRIANOU_RealselfsimilarprocessesStartedfromOrigin2017}. 
Hereby and in the following, $\diag (a_j)$ denotes a diagonal matrix with entries $a_j,j=1,\dots, n$, ``$\circ$'' means elementwise multiplication (of matrices) and $\psi_j(w)=\log \EE[\ee^{wX_1^{(j)}}]$ is the Laplace exponent corresponding to the L\'evy process $X^{(j)}$.\\
The matrix exponent determines the characteristic function of the additive component $X$ since, cf. \cite[Sec. A.1]{DEREICH+DOERING+KYPRIANOU_RealselfsimilarprocessesStartedfromOrigin2017}, 
\begin{equation}\label{eq:lap.tran.map}
	\EE_i\left[\ee^{wX_t}\mathds{1}_{\{J_t=j\}}\right]=\fat{e}_j^\top\ee^{t\fat{\Psi}_X(w)}\fat{e}_i.
\end{equation}

\subsection{Related processes}

As explained in \cite{BEHME+SIDERIS_MMGOU}, for any bivariate $\FF$-MAP $((\zeta,\chi),J)$ the processes $(\zeta+\chi,J)= ((\zeta_t+\chi_t),J_t)_{t\geq 0}$ and $([\zeta,\chi], J)= ([\zeta,\chi]_t,J_t)_{t\geq 0}$  are again $\FF$-MAPs,  where $[\zeta,\chi]$ denotes the co-variation of the two semimartingles $\zeta$ and $\chi$. In particular, the additive component of $([\zeta,\chi], J)$ is given by, see \cite[Eq.\ (2.4)]{BEHME+SIDERIS_MMGOU}
\begin{equation}\label{MAPpathdescriptionsquarebracket}
	[\zeta,\chi]_t
	= \int_{(0,t]}   d [\zeta_{\cdot}^{(J_{\cdot})},\chi_{\cdot}^{(J_{\cdot})}]_{s}+ \sum_{n\geq 1}\sum_{i,j\in S} Z^{ij}_{\zeta,n}Z^{ij}_{\chi,n}\mathds{1}_{\lbrace J_{T_{n-1}}=i,J_{T_{n}}=j,T_n\leq t\rbrace}, \quad t\geq 0. 
\end{equation}

We also recall the \emph{dual} MAP $(X^*,J^*)$ of $(X,J)$, sometimes also called its \emph{time-reversal}, which is defined as the process satisfying 
 \begin{align}\label{DualityRelationMAPs}
 	(X_{(t-s)-}-X_{t},J_{(t-s)-})_{0\leq s\leq t} \text{ under }\PP_\pi \text{ is equal in law to } (X^*_s,J^*_s)_{0\leq s\leq t}  \text{ under }\PP^\ast_\pi,
 \end{align}
where $\fatpi^\ast=\fatpi$ is the unique stationary distribution of $J^{\ast}$, $\PP^\ast_j(\cdot)=\PP(\cdot|J_0^\ast=j)$ and $\PP_\pi^\ast$ defined accordingly; see \cite[Lem. 21]{DEREICH+DOERING+KYPRIANOU_RealselfsimilarprocessesStartedfromOrigin2017} for details. Moreover, note that by \cite[Sec. A.2]{DEREICH+DOERING+KYPRIANOU_RealselfsimilarprocessesStartedfromOrigin2017}, the matrix exponent $\fat{\Psi}_{X^\ast}$ of the dual process $(X^*,J^*)$ fulfills the relation 
\begin{equation}\label{eq:lap.tran.dual}
	\fat{\Psi}_{X^*}(w)=(\diag\fatpi)^{-1}\cdot\fat{\Psi}_X(-w)^\top\cdot(\diag\fatpi),
\end{equation}
for all $w$ where the matrix exponents are defined. \\We denote the return and exit times of the dual process by $\taure_n(j)^*$ and $\tauex_n(j)^*$, $n\in\NN$, $j\in S$.

\normal

\normal
\section{Moments of (integrals with respect to) MAPs} \label{S3} 
\setcounter{equation}{0}

Throughout this section let $(X,J)$ denote a MAP on $\RR\times S$ with $|S|<\infty$, an ergodic Markovian component $J$, and corresponding localization process $(\fatLambda_t)_{t\geq 0}$. Using Perron-Frobenius eigenvalue theory it is possible to determine the moments of the additive component process via the matrix exponent $\fat{\Psi}_X$. Indeed, cf. \cite[Sec. I.6]{ASMUSSEN_AppliedProbandQueues}, the matrix exponent $\fat\Psi_X(w)$ has a real eigenvalue with maximal real part that we denote by $\lambda^X_{\max}(w)$. As shown in \cite[Prop.'s 3.2 and 3.4]{KKPW_HittingTimeofZero}, the eigenvalue $\lambda_{\max}^X$ is simple, and the mapping $w\mapsto \lambda_{\max}^X(w)$ is differentiable and convex. 
Then, e.g., assuming $\Ebb_\pi[|X_1|]<\infty$,  the first moment of $X_1$ follows from {\pa \cite[Cor.\ XI.2.5]{ASMUSSEN_AppliedProbandQueues} as}
$$\EE_\pi[X_1]=\left.\frac{\rmd}{\rmd w}\lambda_{\max}^X(w)\right|_{w=0}.$$
In this paper, we aim to circumvent the eigenvalue theory and use the Lévy system of a MAP instead, to derive moment formulas \anita of MAPs as well as stochastic integrals with respect to MAPs \normal that directly depend on the  semimartingale characteristics of the processes. \anita Such formulas for the mean and variance of the additive component of $(X,J)$ will be presented in Theorems \ref{MAPmeanTheorem} and \ref{Thm-MAPvariance} below, after establishing conditions for finiteness of moments in the upcoming subsection. 

\subsection{Existence of moments}\label{S3a}
\normal

 In \cite[Thm. 34]{DEREICH+DOERING+KYPRIANOU_RealselfsimilarprocessesStartedfromOrigin2017} conditions for the existence of the mean of the additive component of a MAP are collected. The following lemma generalizes parts of these results to general $\kappa$'th moments, $\kappa\geq 1$. \anita Its proof is postponed to Section \ref{Sproofs}. \normal 
 
\begin{lem}\label{prop:mom.map.lev.ju}
	Let $(X,J)$ be an $\FF$-MAP and fix $\kappa\geq 1$, $t>0$. Then the following are equivalent:
	\begin{enumerate}
		\item $\EE_\pi[\sup_{0<s\leq t}|X_s|^\kappa]<\infty$.
		\item $\EE_\pi[|X_t|^\kappa]<\infty$.
		\item $\EE_j[|X_t|^\kappa]<\infty\quad\forall j\in S$.
		\item $\EE[|X_t^{(j)}|^\kappa]<\infty$ for all $j\in S$, and $\EE[|Z^{ij}_{X,1}|^\kappa]<\infty$ for all $(i,j)\in S^2$ such that $q_{ij}>0$.
	\end{enumerate}
\end{lem}

The next {\pa lemma} provides conditions for the existence of exponential moments of the additive component of $(X,J)$, that will later be useful. \anita Again its proof is given in Section~\ref{Sproofs}. \normal

\begin{lem}\label{LemmaIntegrabilitySupremumMAP}
	Let $(X,J)$ be a MAP on $\RR\times S$ and fix $\kappa>0$. Assume that \begin{equation*} \int_{|x|\geq 1} \ee^{\kappa |x|}\nu_{X^{(j)}}(\rmd x)<\infty \text{ for all $j\in S$, and } \EE\Big[\ee^{\kappa |Z^{ij}_{X,1}|}\Big]<\infty \text{ for all }(i,j)\in S^2 \text{ s.t. } q_{ij}>0.\end{equation*}
	Then $\EE_\pi\big[\sup_{s\in[0,1]}\ee^{\kappa |X_s|}\big]<\infty$ and equivalently $\EE_\pi\big[\sup_{s\in[0,t]}\ee^{\kappa |X_s|}\big]<\infty$ for all  $0<t<\infty$.
\end{lem}

Further insight {\pa into} the exponential moments of a MAP can also be obtained from the matrix exponent and its eigenvalue with maximal real part. This is illustrated by the following proposition \anita whose proof is to be found in Section \ref{Sproofs}.\normal

\begin{prop}\label{detPsiEEexplessthan1}
	Let $(X,J)$ be a MAP on $\RR\times S$ and choose $\kappa\in\RR$ such that $\psi_{X^{(j)}}(\kappa)<|q_{jj}|$ for all $j\in S$, and  $\EE[\exp(\kappa Z^{ij}_{X,1})]<\infty$ for all $(i,j)\in S^2$ with $q_{ij}\neq 0$. 
	\begin{enumerate}
		\item If for some $j\in S$
		\begin{equation} \max_{i\in S, i\neq j}\bigg( \frac{1}{|q_{ii}|-\psi_{X^{(i)}}(\kappa)} \sum_{\substack{\ell \in S\\ \ell\neq i,j}} q_{i\ell} \EE\Big[\ee^{\kappa Z_{X,1}^{i\ell}} \Big]\bigg)<1,\label{eq-condmomentreturn} \end{equation}
		then $\EE_j\big[\ee^{\kappa X_{\taure_1(j)}}\big]<\infty$.
		\item If $\EE_j\big[\ee^{\kappa X_{\taure_1(j)}}\big]<\infty$ for some $j\in S$, and  $\lambda_{\max}^X(\kappa)<0$, then  $\EE_j\big[\ee^{\kappa X_{\taure_1(j)}}\big]<1$.
	\end{enumerate} 
\end{prop}

	For our study of moments of the MMGOU processes in Section \ref{S4} below, it will be crucial to study moments of stochastic integrals with respect to MAPs up to a certain stopping time. To this aim we continue by showing finiteness of moments of the stochastic integral $\int_{(0,\tau]} H_{s-}\, \rmd X_s$ for an $\Rbb\times S$-valued $\FF$-MAP  $(X,J)$, a c\`adl\`ag $\Fbb$-adapted process $H$, and a stopping time $\tau$. We make the following assumption.

\begin{ass}\label{ass:int.ass}
	The stopping time $\tau$ has finite moments of every order, that is, $\Ebb[\tau\p n]<\infty$ for every $n\in\mathbb{N}$, and the c\`adl\`ag $\Fbb$-adapted process $H$ satisfies 
	\[
	\EE\Big[\sup_{0\leq t\leq \tau}|H_t|^{\kappa+\varepsilon}\Big]<\infty, \quad \kappa\geq1,
	\]
	for $\varepsilon=0$ if $\tau$ is bounded, and for some $\varepsilon >0$ otherwise.
\end{ass}
If $H$ and $\tau$ satisfy  Assumption \ref{ass:int.ass} but $\tau$ is not bounded, by H\"older's inequality with exponents $p=(\kappa+\varepsilon)/\kappa$  and exponent $q=p/(p-1)= (\kappa+\varepsilon)/\varepsilon$, we immediately derive the estimate
\begin{equation}\label{fun.est.mom}
	\EE\Big[\tau\p\eta \sup_{0\leq t\leq \tau}|H_t|^{\kappa}\Big]\leq \EE\Big[\sup_{0\leq t\leq \tau}|H_t|^{\kappa+\varepsilon}\Big]\p{1/p}\Ebb\big[\tau\p{\eta q} \big]\p{1/q}<\infty,\quad \eta>0.
\end{equation}

The next theorem is the main result of this section:
\begin{theorem}\label{Theo_mean_MAP_integral}
	Let $\kappa\geq1$.  If $(X,J)$ is an $\FF$-MAP on $\RR\times S$ such that $\Ebb_\pi[|X_1|\p\kappa]<\infty$ and $\tau$ and $H$ satisfy Assumption \ref{ass:int.ass}, then for all $j\in S$
	\[
	\EE_j\bigg[\sup_{0\leq t\leq \tau}\bigg|\int_{(0,t]} H_{s-}\rmd X_s\bigg|^\kappa\bigg]<\infty.
	\]
\end{theorem}

To prove Theorem \ref{Theo_mean_MAP_integral} we use the decomposition $X=X_1+X_2$ of the MAP $X$ given in \eqref{MAPpathdescription} and consider the integrals with respect to $X_1$ and $X_2$ separately. As $X_1$ is a concatenation of Lévy processes, \anita our studies of this part rely on \normal a generalization of \cite[Lem.\ 6.1]{BEHME_DistributionalPropertiesGOULevynoise} where moments of stochastic integrals with respect to Lévy processes up to a fixed time have been considered. The proof of \cite[Lem.\ 6.1]{BEHME_DistributionalPropertiesGOULevynoise} could not be generalized
to integration up to a general stopping time and \anita in Section \ref{Sproofs} below we therefore provide an alternative proof.\normal

\begin{lem}\label{Theoremn_Mean_Stopped_LP_integral}
	Let $\kappa \geq1$. If $X$ is a L\'evy process such that $\EE[|X_1|^{\kappa}]<\infty$ and $\tau$ and $H$ satisfy Assumption \ref{ass:int.ass}, then
	\[
	\EE\bigg[\sup_{0\leq t\leq \tau}\bigg|\int_{(0,t]} H_{s-}\rmd X_s\bigg|^\kappa\bigg]<\infty.
	\]
\end{lem}

\anita 
In order to deal with the integral with respect to the additional jumps of $X$, \normal following \cite{PSS18}, observe that we have the representation 
\begin{equation}\label{eq:def.Psii}
	X_{2,t}=\sum_{\substack{i,j\in S\\ i\neq j}} \Psi\p{ij}_t,\qquad\Psi\p{ij}_t=\int_{(0,t]} \int_\Rbb x\,\mu_Z\p {ij}(\rmd s,\rmd x)
\end{equation}
with the integer-valued random measures $\mu_Z\p {ij}$, $i,j\in S, i\neq j,$ defined by 
\[
\mu_Z\p {ij}(\om,\rmd t,\rmd x)=\sum_{n\geq1}\mathds{1}_{\{J_{T_n-}(\omega)=i, J_{T_n}(\omega)=j\}}\delta_{(T_n(\om),Z\p{ij}_{X,n}(\om))}(\rmd t,\rmd x)\mathds{1}_{\{T_n(\omega)<\infty\}}.
\]
We stress that the definition of $\mu_Z\p {ij}$ in \cite{PSS18} ($\Pi_Z\p i$ in the notation used in \cite{PSS18}) is flawed: In \cite[Eq.\ (6)]{PSS18} the Dirac measure $\delta_{Z\p i_n}(\rmd x)$ is replaced by $\Pbb_{Z_n\p i}(\rmd x)$. This however is inconsistent with \eqref{eq:def.Psii}. 

In the proof of Theorem \ref{Theo_mean_MAP_integral} we will apply  \cite[Thm.\ 1]{MaRoe14} to the integral with respect to $X_2$. Therefore, we need the explicit form of the $\Fbb$-dual predictable projection $\nu_Z\p {ij}$ of $\mu_Z\p{ij}$, which we are now going to compute. Recall that ${\pa F\p{ij}_X}$ denotes the cdf of the law of $Z\p{ij}_{X,1}$, and that $N^{i\to j}_t:=\mu_Z\p{ij}([0,t]\times\Rbb)$ denotes the number 
of jumps of $J$ from state $i$ to state $j$ up to $t\geq0$. By \cite[Appendix B, p.\ 361]{ELLIOG+AGGOUN+MOORE_HiddenMarkovModelsEstimationControl1995} {\pa (see also \cite[p.\ 290]{ZESG2011}) the process  $N^{i\to j}-\int_{(0,\cdot]} \mathds{1}_{\{J_{s-}=i\}}q_{ij}\rmd s$ is a martingale, hence the $\Fbb$-dual predictable projection $\phi\p{ij}$ of $N^{i\to j}$} is given by 
\begin{equation} \label{eq-projectNitoj}
	\phi_t\p{ij}=\int_{(0,t]} \mathds{1}_{\{J_{s-}=i\}}q_{ij}\rmd s, \quad t\geq 0.
\end{equation} 
From this we derive the following Lemma, see also \cite[Sec. 2]{KyprianouEntrance} and \cite{SalahMorales} for related results and special cases. \anita Its proof is given in Section \ref{Sproofs}. \normal

\begin{lem}\label{lem:com.Psii}
	For any $i,j\in S$, $i\neq j$, the $\Fbb$-dual predictable projection $\nu_Z\p {ij}$ of $\mu_Z\p {ij}$ is given by
	\[
	\nu\p {ij}_Z(\om,\rmd t,\rmd x)={\pa F_X\p{ij}}(\rmd x)\mathds{1}_{\{J_{t-}(\om)=i\}}q_{ij}\rmd t.
	\]
\end{lem}

We are now ready to prove Theorem \ref{Theo_mean_MAP_integral}. For this recall that by Jensen's inequality for any non-negative numbers $a_1,\ldots, a_n$ and $\kappa\geq 1$,  we have
\begin{equation}\label{jensenestimate}
	\bigg( \sum_{i=1}^n a_i	\bigg)^\kappa \leq n^{\kappa-1} \sum_{i=1}^n a_i^\kappa.
\end{equation}

\begin{proof}[Proof of Theorem \ref{Theo_mean_MAP_integral}] By \eqref{MAPpathdescription} and Minkowski's inequality we get
	\begin{align}
		\lefteqn{\EE_j\bigg[\sup_{0\leq t\leq \tau}\bigg|\int_{(0,t]} H_{s-}\rmd X_s\bigg|^\kappa\bigg]^\frac{1}{\kappa} \nonumber}\\
		&\leq
		\EE_j\bigg[\sup_{0\leq t\leq \tau}\bigg|\sum_{i\in S}\int_{(0,t]} H_{s-}\mathds{1}_{\lbrace J_{s-}=i\rbrace}\rmd  X_s^{(i)}\bigg|^\kappa\bigg]^\frac{1}{\kappa} +\EE_j\bigg[\sup_{0\leq t\leq \tau}\bigg|\int_{(0,t]} H_{s-}\rmd  X_{2,s}\bigg|^\kappa\bigg]^\frac{1}{\kappa},\label{Theorem_Mean_Stopped_MAPintegral}
	\end{align}
	where for the first summand an application of \eqref{jensenestimate}  yields 
	\begin{equation}\label{Theorem_Stopped_integral_MAP_eq:first part}
		\EE_j\bigg[\sup_{0\leq t\leq \tau}\bigg	|\sum_{i\in S}\int_{(0,t]}  H_{s-}\mathds{1}_{\lbrace J_{s-}=i\rbrace}\rmd  X_s^{(i)}\bigg|^\kappa\bigg]  \leq
		|S|^{\kappa-1}\sum_{i\in S}\EE_j\bigg[\sup_{0\leq t\leq \tau}\bigg|\int_{(0,t]} H_{s-}\mathds{1}_{\lbrace J_{s-}=i\rbrace}\rmd X_s^{(i)}\bigg|^\kappa\bigg]. 
	\end{equation}
	As, by the assumptions of the theorem, we have $\EE[|X_1^{(j)}|^\kappa]<\infty$ for all $j\in S$ by Lemma \ref{prop:mom.map.lev.ju}, and the process $H\mathds{1}_{\lbrace J=i\rbrace}$ satisfies Assumption \ref{ass:int.ass}, this is finite by Lemma \ref{Theoremn_Mean_Stopped_LP_integral}. \\
	We now consider the second summand in \eqref{Theorem_Mean_Stopped_MAPintegral} and note that by Lemma \ref{prop:mom.map.lev.ju}, from our assumptions, we have $\EE[|Z\p{ij}_{X,1}|^\kappa]<\infty$ for all $(i,j)\in S^2$ such that $q_{ij}>0$. From \eqref{eq:def.Psii} and \eqref{jensenestimate} it follows
	\begin{align}
		\lefteqn{\Ebb_j\bigg[\sup_{0\leq t\leq \tau}\bigg|\int_{(0,t]} H_{s-}\rmd X_{2,s}\bigg|\p\kappa\bigg] = \Ebb_j\bigg[\sup_{0\leq t\leq \tau}\bigg|\sum_{\substack{i,k\in S\\ i\neq k}} \int_{(0,t]} H_{s-}\rmd \Psi^{ik}_s \bigg|\p\kappa\bigg] } \nonumber \\
		&\leq
		(|S|^2-|S|)^{\kappa-1}\sum_{\substack{i,k\in S\\ i\neq k}}\Ebb_j\bigg[\sup_{0\leq t\leq \tau}\bigg|\int_{(0,t]} H_{s-}\rmd \Psi\p{ik}_s\bigg|\p\kappa\bigg]\nonumber 
		\\
		\begin{split} \label{eq:est.in.add.ju}
			&\leq (2(|S|^2-|S|))^{\kappa-1}\sum_{\substack{i,k\in S\\ i\neq k}}\bigg(\Ebb_j\bigg[\sup_{0\leq t\leq \tau}\bigg|\int_{(0,t]} \int_\Rbb H_{s-}x\,\overline \mu\p{ik}_Z(\rmd s,\rmd x)\bigg|\p\kappa\bigg]\\&\hspace{6cm}+\Ebb_j\bigg[\sup_{0\leq t\leq\tau}\bigg|\int_{(0,t]} \int_\Rbb H_{s-}x\,\nu\p{ik}_Z(\rmd s,\rmd x)\bigg|\p\kappa\bigg]\bigg), \end{split}
	\end{align} 
	where we define $\overline \mu\p{ij}_Z:=\mu\p{ij}_Z-\nu\p{ij}_Z$ with $\nu\p{ij}_Z$ as in Lemma \ref{lem:com.Psii}. Hereby, due to the specific form of $\nu\p{ij}_Z$, we get
	\begin{align}\nonumber
		\Ebb_j\bigg[\sup_{0\leq t\leq \tau}\bigg|\int_{(0,t]} \int_\Rbb H_{s-}x\,\nu\p{ik}_Z(\rmd s,\rmd x)\bigg|\p\kappa\bigg]&\leq
		\Big(q_{ik}\Ebb\big[|Z\p{ik}_{X,1}|\big]\Big)\p\kappa\Ebb_j\bigg[\sup_{0\leq t\leq \tau}\bigg|\int_{(0,t]} H_{s-}\mathds{1}_{\lbrace J_{s-}=i\rbrace}\rmd s\bigg|\p\kappa\bigg]
		\\&\leq
		\Big(q_{ik}\Ebb\big[|Z\p{ik}_{X,1}|\big]\Big)\p\kappa{\pa\Ebb_j}\Big[\tau\p\kappa \sup_{0\leq t\leq\tau}|H_t|\p\kappa\Big]<\infty,\label{eq:est.n}
	\end{align}
	where the last estimate follows by \eqref{fun.est.mom} with $\eta=\kappa$, and with $\varepsilon>0$ if $\tau$ is not bounded. This shows that the second expectation in \eqref{eq:est.in.add.ju} is finite for all $i,k\in S$, $i\neq k$.   Concerning the first expectation in \eqref{eq:est.in.add.ju}, using \cite[Thm.\ 1]{MaRoe14}, we get
	\begin{align}
		\lefteqn{\Ebb_j\bigg[\sup_{0\leq t\leq\tau}\bigg|\int_{(0,t]} \int_\Rbb H_{s-}x\,\overline{\mu}_Z\p{ik} (\rmd s,\rmd x)\bigg|\p\kappa\bigg] =\Ebb_j\bigg[\sup_{t\geq 0}\bigg|\int_{(0,t]}\int_\Rbb H_{s-} \mathds{1}_{\{s\leq \tau\}} x\,\big(\mu_Z\p{ik}-\nu\p{ik}_Z\big)(\rmd s,\rmd x)\bigg|\p\kappa\bigg] \nonumber }\\ 
		&\leq \begin{cases}
			\displaystyle C\p1_\kappa\,\Ebb_j\bigg[\int_{(0,\tau]} \int_\Rbb|H_{s-}x|\p\kappa\nu\p{ik}_Z(\rmd x)\rmd s \bigg],&  \kappa\in[1,2)\\ 
			\displaystyle C\p2_\kappa\,\bigg(\Ebb_j\bigg[\bigg(\int_{(0,\tau]} \int_\Rbb |H_{s-}x|\p2\nu\p {ik}_Z( \rmd s,\rmd x)\bigg)\p{\frac{\kappa}{2}}\bigg]+\Ebb_j\bigg[\int_{(0,\tau]} \int_\Rbb |H_{s-}x|\p\kappa\nu\p {ik}_Z(\rmd s,\rmd x)  \bigg]\bigg),& \kappa\in[2,\infty),
		\end{cases}
		\label{eq:est.exp.add.ju.com} 
	\end{align}
	for some constants \pa $C_\kappa^1, C_\kappa^2\in(0,\infty)$. \normal
	Hence, using \eqref{fun.est.mom} with $\eta=1$ or $\eta=\kappa/2$, and with $\varepsilon>0$ if $\tau$ is not bounded,  we get as in \eqref{eq:est.n}
	\[
	\Ebb_j\bigg[\sup_{0\leq t\leq\tau}\bigg|\int_{(0,t]} \int_\Rbb H_{s-}x\,\overline{\mu}_Z\p{ij} (\rmd s,\rmd x)\bigg|\p\kappa\bigg]<\infty. 
	\]
	Therefore, the right-hand side of \eqref{eq:est.in.add.ju} and hence of \eqref{Theorem_Mean_Stopped_MAPintegral} is finite, and the proof of the theorem is complete.
\end{proof}

\anita
\subsection{Explicit expressions for the moments}\label{S3b}
\normal 

We can now provide an explicit formula for the mean of the additive component of a MAP. \anita The proof of this theorem is given in Section \ref{Sproofs}. \normal 

\begin{theorem}\label{MAPmeanTheorem}
	Let $(X,J)$ be an $\FF$-MAP on $\Rbb\times S$. Suppose $\EE_j[|X_1|]<\infty$ for all $j\in S$, then for all $j\in S$
	\begin{equation}\label{eq:meanhatfatX}
		\EE_j\big[\fat{\hat{X}}_t\big]	= \int_0^t\ee^{\fatQ^\top(t-s)}\fateps[X]\ee^{\fatQ^\top s}\rmd s\cdot \fat e_j,
	\end{equation}
	and in particular
	\begin{align}\label{MAPmeanTheoremintegralformula}
		\EE_j [X_t]=\fatone^\top\fateps[X]\int_0^t \ee^{\fatQ^\top s}\rmd s\cdot\fat{e_j}, \quad\text{and} \quad  \EE_\pi [X_t] = \fatone^\top\fateps[X]\int_0^t \ee^{\fatQ^\top s}\rmd s\cdot\fatpi,
	\end{align}
	with the expectation matrix 
	\begin{align}\label{eq-expmatrix}
		\fateps[X]:= \diag\left(\EE\big[X_1^{(j)}\big]\right)+\fatQ^\top\circ\left(\EE\big[Z^{ij}_{X,1}\big]\right)^\top_{i,j\in S}.
	\end{align}
	Moreover, $\fat{X}^\#:=\fat{\hat{X}}-\fateps[X]\int_{(0,\cdot]} \fatLambda_s \rmd s$ is {\pa an $\Fbb$-martingale under $\Pbb$}.
\end{theorem}

The following representation of the expectation of a MAP could also have been derived from \cite[Cor. XI.2.5]{ASMUSSEN_AppliedProbandQueues}. We mention it here, as the proof in that source is not fully given. It is an immediate consequence of \eqref{MAPmeanTheoremintegralformula} upon noticing that $\ee^{\fatQ^\top s}\cdot\fatpi = \fatpi$ for all $s\geq 0$, and due to \cite[Eq. (25.7)]{SATO_LPinfinitelydivisibledistributions}.  

\begin{cor}
		Let $(X,J)$ be an $\FF$-MAP on $\Rbb\times S$ and suppose that $\EE_j[|X_1|]<\infty$ for all $j\in S$, then
		$${\pa\EE_\pi[X_1]} = \sum_{j\in S} \pi_j \bigg(\gamma_{X^{(j)}} + \int_{|x|>1} x \nu_{X^{(j)}}(\rmd x)  + \sum_{\substack{i\in S\\i\neq j}}  q_{ij} \int_\RR x F^{ij}_X (\rmd x) \bigg).$$
\end{cor}

\anita We continue by providing an explicit expression of the variance of the additive component of a MAP in the next theorem. \normal

\begin{theorem}\label{Thm-MAPvariance}
	Let $(X,J)$ be an $\FF$-MAP on  $\RR\times S$. Suppose that  $\EE_j[|X_1|^2]<\infty$ for all $j\in S$, then for all $j\in S$, $t>0$
	\begin{align}\label{eq:2ndmomentMAP}
		\EE_j[X_t^2 \fatLambda_t]&= \bigg(2 \int_0^t \ee^{\fatQ^\top(t-s)}  \fateps[X] \int_0^s \ee^{\fatQ^\top(s-u)} \fateps[X] \ee^{\fatQ^\top u} \rmd u \rmd s  + \int_0^t \ee^{\fatQ^\top(t-s)} \fateps\big[[X,X]\big] \ee^{\fatQ^\top s} \rmd s\bigg) \cdot \fat{e}_j, 
	\end{align}
	with $\fateps\big[[X,X] \big] =\diag\left(\var (X^{(j)}_1)\right)+\fatQ^\top\circ\left(\EE\left[(Z^{ij})^2\right]\right)^\top_{i,j\in s}$.
	In particular this implies
	\begin{align*}
		\Var_j[X_t]&=\fatone^\top\fateps[X]\left(2\int_0^t\int_0^s\ee^{\fatQ^\top(s-u)}\fateps[X]\ee^{\fatQ^\top u}\rmd u\rmd s-\int_0^t\ee^{\fatQ^\top s}\rmd s \cdot \fat e_j\cdot \fatone^\top\fateps[X]\int_0^t\ee^{\fatQ^\top s}\rmd s\right)\fat e_j
		\\&
		+\fatone^\top\fateps\big[[X,X] \big] \int_0^t\ee^{\fatQ^\top s}\rmd s \cdot \fat e_j.
	\end{align*}
\end{theorem}

\anita 
The proof of Theorem \ref{Thm-MAPvariance} as given in Section \ref{Sproofs} relies on an application of integration by parts and the upcoming Lemma \ref{lem:Mean_MAPintegral} which  
\normal will also be useful for our computations in Section \ref{S4}. \anita Its proof can also be found in Section \ref{Sproofs}. \normal

\begin{lem}\label{lem:Mean_MAPintegral}
	Let $(X,J)$ be an $\FF$-MAP on $\RR\times S$ and $H$ an $\FF$-adapted càdlàg process. Assume that $\EE_\pi [|X_1|]<\infty$ and fix $t>0$ such that $\EE[\sup_{0<s\leq t}|H_s|]<\infty$. Then for all $j\in S$ it holds
	\begin{equation}\label{eq:Theorem_Mean_Time_MAPintegral}
		\EE_j\bigg[\int_{(0,t]}H_{s-}\rmd X_s\bigg]=\fatone^\top\fateps[X]\int_{(0,t]} \EE_j\big[\hatfat{H}_s\big]\rmd s,
	\end{equation}
	and 
	\begin{equation}\EE_j\bigg[\int_{(0,t]}\hatfat H_{s-}\rmd X_s\bigg]+\EE_j\bigg[\int_{(0,t]}H_{s-}\rmd [X,\fatLambda]_s\bigg]= 	\fateps[X]\int_{(0,t]}\EE_j\big[\hatfat{H}_s\big]\rmd s .\label{eq:Theorem_Mean_Time_MAPintegral_vectorvalued}\end{equation}
 with the expectation matrix $\fateps[X]$ defined in \eqref{eq-expmatrix}.
\end{lem}

\begin{remark}
Via Itô's formula and similar computations as \anita in the proof of Theorem \ref{Thm-MAPvariance} \normal one can also derive closed form expressions for higher moments of the additive component of a MAP, which are then expressed in terms of powers of the expectation matrix, the correlation matrix $\diag (\var(X^{(j)}))$, and multiple integrals with integrand $\ee^{\fatQ^\top t}$. In order to avoid overly lengthy computations we {\pa restrain} from giving any details, \pa but instead also refer to the recent work \cite{assmussenbladt} where recursive expressions of moments of MAPs similar to the above are presented in Prop. 3.\normal 
\end{remark}

\section{\anita Application to  Markov modulated GOU processes} \label{S4} 
\setcounter{equation}{0}
\normal

In this final section we consider the MMGOU process $(V_t)_{t\geq 0}$  solving the SDE \eqref{MMGOUSDE} for some bivariate MAP $((U,L),J)$. 
In Section \ref{S4a} we compute the running mean and the autocovariance function of the MMGOU process under suitable conditions that imply their existence. Afterwards, Section \ref{S4b} focuses on stationary MMGOU processes: We derive suitable conditions for the existence of the moments of the stationary distribution, as well as explicit moment formulas. Similar results as in this section in the special case of the MMOU process have been obtained in \cite{HUANG+SPREJ_MarkovmodulatedOUP2014} and \cite{LINDSKOG+MAJUMDER_ExactLongtimebehaviourMAP2019}. However, in the latter, no expressions in terms of the driving processes are provided. For the special case of the GOU process, moments of the stationary distribution have been considered in \cite{LINDNER+MALLER_LevyintegralsStationarityGOUP} (existence) and  \cite{BEHME_DistributionalPropertiesGOULevynoise} (formulas). Related results can also be found in \cite{XING+ZHANG+WANG_StationaryDistributionTwoClassesReflcetedOUP2009}, where moments of the stationary distribution of reflected Markov modulated OU processes are computed, and in \cite{SALMINENVOSTRIKOVA} where moments of exponential functionals of additive processes are studied.

Throughout, we assume $\Delta U>-1$. This allows us to use the explicit representation \eqref{MMGOUexplicit} of the MMGOU process with the MAP $((\xi,\eta),J)$ that is  one-to-one related with $((U,L),J)$ by
\begin{equation}\label{eq-ULviaxieta}
	\begin{pmatrix}
		U_t \\ L_t
	\end{pmatrix} = \begin{pmatrix}
		-\xi_t + \frac12 \int_{(0,t]} \sigma_\xi^2(J_s) \rmd s + \sum_{0<s\leq t} \left( \Delta \xi_s + \ee^{-\Delta \xi_s} - 1 \right) \\ \eta_t -\int_{(0,t]} \sigma_{\xi,\eta}(J_s) \rmd s +\sum_{0<s\leq t} \left(\ee^{-\Delta\xi_s}-1\right) \Delta\eta_s
	\end{pmatrix}, \quad t\geq 0,
\end{equation}
as shown in \cite[Prop.\ 2.11]{BEHME+SIDERIS_MMGOU}. Note that in absence of additional jumps, {\pa that is if $Z_{U,n}\p{ij}\equiv Z_{L,n}\p{ij}\equiv0$,} the conditional independence of $\xi$ and $\eta$ (or analogously of $U$ and $L$) given $J$ implies $L_t=\eta_t$ for all $t\geq 0$. Furthermore, the relation between $U$ and $\xi$ in \eqref{eq-ULviaxieta} is equivalent to
\begin{equation} 
\label{eq_Uandxi} \ee^{-\xi_t} = \mathcal{E}(U)_t, \quad t\geq 0, 
\end{equation}
where $(\mathcal{E}(U)_t)_{t\geq 0}$ is the \emph{(Dol\'eans-Dade) stochastic exponential} of $U$, i.e. the unique solution of the SDE $\rmd Z_t = Z_{t-} \rmd U_t$, $t\geq 0$ with $Z_0=\mathcal{E}(U)_0=1$. This intimate relationship between $U$ and $\xi$ implies that their moments are also closely related. For $|S|=1$, i.e. for $U$ and $\xi$ being Lévy processes, this has been elaborated in \cite[Prop. 3.1]{BEHME_DistributionalPropertiesGOULevynoise}. The upcoming two results consider the general case $|S|\geq 1$. 

\begin{lem}\label{Relation(U,xi)integrability} Let $(\xi,J)$ be a MAP and define the MAP $(U,J)$ via \eqref{eq_Uandxi}. Then, for any $\kappa \geq 1$, we have
$$\EE_\pi\big[|U_1|^\kappa\big]<\infty\quad \text{ if and only if } \quad \EE_\pi\big[\ee^{-\kappa\xi_1}\big]<\infty.$$
\end{lem}
\begin{proof}
	The fact that $(\xi,J)$ is a MAP if and only if $(U,J)$ is a MAP with $\Delta U >-1$ has been shown in \cite{BEHME+SIDERIS_MMGOU}. Further, by Lemma \ref{prop:mom.map.lev.ju} we know that $\EE_\pi[|U_1|^\kappa]<\infty$ if and only if $\EE[|U_1^{(j)}|^\kappa]<\infty$ for all $j\in S$ and $\EE[|Z_{U,1}^{ij}|^\kappa]<\infty$ for all $(i,j)\in S^2$ such that $q_{ij}>0$. 
	The L\'evy processes $U^{(j)}$, $j\in S$, have finite $\kappa$'s moment if and only if $\int_{|x|>1}|x|^\kappa \nu_{U^{(j)}}(\rmd x)<\infty$, $j\in S$, cf.  \cite[Thm. 25.17]{SATO_LPinfinitelydivisibledistributions}. Via the path decomposition \eqref{MAPpathdescription} we observe that \eqref{eq-ULviaxieta} implies 
	$$U_t^{(j)} =-\xi_t^{(j)} +\frac{1}{2}\int_{(0,t]}\sigma\p 2(j)\rmd s+\sum_{0<s\leq t}\left(\ee^{-\Delta\xi^{(j)}_s}-1+\Delta\xi^{(j)}_s\right), \quad \text{and} \quad Z_{U,1}^{ij} = \ee^{-Z_{\xi,1}^{ij}}-1,$$ 
	and hence $\EE[|U_1^{(j)}|^\kappa]<\infty$ is equivalent to $\int_{|x|>1}\ee^{-\kappa x}\nu_{\xi^{(j)}}(\rmd x)<\infty$, and hence to $\EE[e^{-\kappa \xi_1^{(j)}}]<\infty$, again by \cite[Thm. 25.17]{SATO_LPinfinitelydivisibledistributions}. Summing up, we observe that  $\EE_\pi[|U_1|^\kappa]<\infty$ if and only if $\EE[e^{-\kappa \xi_1^{(j)}}]<\infty$ for all $j\in S$ and $\EE[\ee^{-\kappa Z_{\xi,1}^{ij}}]<\infty$ for all $(i,j)\in S^2$ such that $q_{ij}>0$. \\
	Finally, by an analogue computation 
		as in \eqref{eq-momentexistencehelp}, we observe that this is in turn equivalent to $\EE_\pi\big[\ee^{-\kappa\xi_1}\big]<\infty$, which finishes the proof.
\end{proof}

\begin{prop}\label{expectationmatrixstochasticexponential}
	Let $(\xi,J)$ be a MAP and define the MAP $(U,J)$ via \eqref{eq_Uandxi}.  Then for any $k\in\NN$ such that $\EE_\pi[\sup_{0<s\leq t} \cE(U)_s^k]= \EE_\pi[\sup_{0<s\leq t} \ee^{-k\xi_s}]<\infty$ for some $t\geq 0$ we have the identity
	\begin{equation}\label{eq:psi.xi.minus.k}
	\fat\Psi_{\xi}(-k)= 	\fatQ^\top+\fateps\bigg[kU + \frac{k(k-1)}{2}\int_{(0,\cdot]} \sigma_U^2(J_s) \rmd s + \sum_{0<s\leq \cdot}((1+\Delta U_s)^k-1-k\Delta U_s)\bigg].
	\end{equation}
	In particular it holds
	\begin{equation}
		\label{eq_invertiblematrices}
		\fat\Psi_{\xi}(-1) = \fatQ^\top+\fateps[U],\quad \text{and} \quad  \fat\Psi_{\xi}(-2)= \fatQ^\top+  2\fateps[U]  + \fateps\big[[U,U]\big].
	\end{equation}
\end{prop}
\begin{proof} On the one hand, as the matrix exponent $\fat{\Psi}_\xi$ fulfills \eqref{eq:lap.tran.map}, relation \eqref{eq_Uandxi} immediately implies
	\begin{equation}\label{eq:lap.exp}
	\EE_\pi\big[\CE(U)_t^k\big]=\EE_\pi\big[\ee^{-k\xi_t}\big]
	=  \sum_{i,j\in S} \pi_i  \EE_i \big[\ee^{-k\xi_t} \mathds{1}_{\{J_t=j\}} \big]= \fat{1}^\top\exp(t\fat\Psi_\xi(-k))\fatpi.  \end{equation}
	On the other hand, from the SDE for the Doléan-Dade stochastic exponential, we have $\Delta \cE(U)= \cE(U)_{-} \Delta U$. Thus we compute via the binomial theorem
	\begin{align*}
		\Delta(\cE(U)_t^k)&= (\cE(U)_{t-}+ \Delta \cE(U)_t)^k- \cE(U)_{t-}^k 
	= \sum_{\ell=0}^{k-1} {k\choose \ell} \cE(U)_{t-}^\ell  (\Delta \cE(U)_t)^{k-\ell} \\
		&= \cE(U)_{t-}^k \cdot \sum_{\ell=0}^{k-1} {k\choose \ell}  (\Delta U_t)^{k-\ell}  = \cE(U)_{t-}^k \left( (1+\Delta U_t)^k -1 \right).
	\end{align*} 
	Moreover, since $\rmd \cE(U)_t\p c=\cE(U)_{t-} \rmd U^c_t$, we have $$\langle \cE(U)^c,\cE(U)^c\rangle_t= [\cE(U),\cE(U)]_t^c=\int_{(0,t]} \cE(U)_{s-}\p2 \rmd \langle U^c,U^c\rangle_s.$$
	Thus It\^o's formula (cf. \cite[Thm. II.32]{PROTTER_StochIntandSDE}) yields
	$$	\cE(U)_t^k=1+ \int_{(0,t]} \cE(U)_{s-}^k \rmd Y_s^{(k)}, $$
	where we define
\begin{align*}Y_t^{(k)}&:=kU_t + \frac{k(k-1)}{2}\langle U^c,U^c\rangle_t + \sum_{0<s\leq t}((1+ \Delta U_s)^k-1-k\Delta U_s)\\
	&= k\xi_t + \frac{k^2}{2} \int_{(0,t]} \sigma_\xi^2(J_s) \rmd s + \sum_{0<s\leq t}(\ee^{-k\Delta \xi_s}-1-k\Delta \xi_s) ,\quad  t\geq0.\end{align*}
	Notice that $(Y^{(k)},J)$ is an $\Fbb$-MAP for any $k$ as can easily be seen by the path description \eqref{MAPpathdescription}.
	Moreover, due to the given form of $Y^{(k)}$, we conclude by Lemmas \ref{prop:mom.map.lev.ju} and \ref{LemmaIntegrabilitySupremumMAP} that our assumptions imply $\EE_\pi[|Y_1^{(k)}|]<\infty$. Thus, via partial integration as in \eqref{integrationbyparts}, applying \eqref{eq:Theorem_Mean_Time_MAPintegral_vectorvalued} and {\pa arguing} via Lemma \ref{eq:unif.mart.int.M}, we obtain
	\begin{align*}
			\EE_\pi \big[\CE(U)_t^k\fatLambda_t\big]
			&	= \EE_\pi \big[\fatLambda_0\big] + \big(\fatQ^\top + 	\fateps[Y^{(k)}]\big)	 \int_{(0,t]} \EE_\pi\big[\CE(U)_{s}^k\fatLambda_{s} \big]\rmd s. 
	\end{align*} 
As $\cE(U)_0=1$, this ODE is solved uniquely by 
\begin{equation*}
	\EE_\pi \big[\CE(U)_t^k\fatLambda_t\big]= \exp\big(t\big( \fatQ^\top + \fateps[Y^{(k)}] \big) \big) \fatpi.
	\end{equation*}
Multiplying with $\fatone^\top$ proves \eqref{eq:psi.xi.minus.k} due to \eqref{eq:lap.exp}. Finally
we get
$\fateps[Y^{(1)}]= \fateps[U]$, while $\fateps[Y^{(2)}]= \fateps[2 U + [U,U]]= 2\fateps[U] + \fateps[[U,U]]$, as $\fateps[\cdot]$ is additive. This proves \eqref{eq_invertiblematrices}. \end{proof}

\subsection{Mean and autocovariance structure of the MMGOU process}\label{S4a}

We start with a technical lemma which provides sufficient conditions for the  existence of moments of the MMGOU process.

\begin{lem} \label{lem-supmomentV}
	Consider the MMGOU process $(V_t)_{t\geq 0}$ driven by the bivariate MAP $((\xi,\eta),J)$. Assume that for some $\kappa \geq 1$ and all $j\in S$
	$$\EE_j\Big[\sup_{0<s\leq 1}\ee^{\kappa |\xi_s|}\Big]<\infty,\quad  \EE_j\big[|V_0|^{\kappa}\big]<\infty,\quad \text{and }\; \EE_j\big[|\eta_1|^{\kappa}\big]<\infty,$$
	then $\EE_j\big[\sup_{0<s\leq t}|V_s|^\kappa\big]<\infty$ for all $j\in S$ and $t\geq 0$. \\
	If $\kappa=2$, then in particular  $\EE_j\left[\sup_{0<s\leq t}|V_s V_t|\right]<\infty$ for all $t\geq 0$.
\end{lem}
\begin{proof}
	Using \eqref{MMGOUexplicit} and the inequality \eqref{jensenestimate} we have
	\begin{align*}
	\EE_j\Big[\sup_{0<s\leq t}|V_s|^\kappa \Big]
		&=
		\EE_j\bigg[\sup_{0<s\leq t}\bigg|\ee^{- \xi_s}\int_{(0,s]}\ee^{\xi_{u-}}\rmd \eta_u+\ee^{-\xi_s}V_0\bigg|^\kappa\bigg]		\\
		&\leq 2^{\kappa-1} \bigg(\EE_j\bigg[\sup_{0<s\leq t}\Big|  \ee^{-\xi_s}\int_{(0,s]}\ee^{\xi_{u-}}\rmd \eta_u\Big|^\kappa \bigg]+ \EE_j\Big[\sup_{0<s\leq t} \ee^{-\kappa \xi_s}\Big] \EE_j\big[ |V_0|^\kappa\big] \bigg),
	\end{align*}
 where  we {\pa used the fact} that $V_0$ is chosen independently of $((\xi_t,\eta_t), J_t)_{t\geq 0}$. While the second summand is finite by assumption and Lemma \ref{LemmaIntegrabilitySupremumMAP}, for the first summand we note that by 
\cite[Lem. 3.1(ii)]{BEHME+SIDERIS_MMGOU} 
$$\ee^{-\xi_s}\int_{(0,s]}\ee^{\xi_{u-}}\rmd \eta_u \text{ under } \PP_\pi \text{ is equal in law to } -\int_{(0,s]} \ee^{\xi^\ast_{u-}} \rmd L^\ast_u \text{ under } \PP^\ast_\pi.$$
As the dual processes $\xi^\ast$ and $L^\ast$ inherit the moments of $\xi$ and $L$, we may conclude via \eqref{eq-ULviaxieta} and the results in Section \ref{S3} that under our conditions $\EE^\ast_\pi [\sup_{0<s\leq t}\ee^{\kappa |\xi^\ast_s|}]<\infty$ and $\EE^\ast_\pi [|L_1|^\kappa]<\infty$. Hence  finiteness of the first summand follows from Theorem  \ref{Theo_mean_MAP_integral} and this implies the statement. \\
Finally, if the assumptions hold for $\kappa=2$, then
	\begin{align*}
		\EE_j\Big[\sup_{0<s\leq t}|V_s V_t|\Big]\leq \EE_j\Big[\sup_{0<s\leq t}|V_s|^2\Big]^\frac{1}{2}\EE_j\left[|V_t|^2\right]^\frac{1}{2},
	\end{align*}	
	is also finite by the above.
\end{proof}

\begin{theorem}\label{Prop_runningmean}
	Consider the MMGOU process $(V_t)_{t\geq 0}$ driven by the bivariate MAP $((\xi,\eta),J)$ and solving the SDE \eqref{MMGOUSDE} for the bivariate MAP $((U,L),J)$ in \eqref{eq-ULviaxieta}. Assume that for all $j\in S$
	$$\EE_j\Big[\sup_{0<s\leq 1}\ee^{|\xi_s|}\Big]<\infty,\quad  \EE_j\big[|V_0|\big]<\infty,\quad \text{and }\; \EE_j\big[|\eta_1|\big]<\infty.$$ Then for all $t\geq 0$ and all $j\in S$
	\begin{equation}\label{eq-GOUrunningmean}
		\EE_j\big[V_t\big]= \fatone^\top
		\ee^{(\fatQ^\top+\fateps[U])t}\EE_j\big[\fat{\hat{V}}_0\big]
		+ \fatone^\top \int_0^t\ee^{(\fatQ^\top
			+\fateps[U])(t-s)}\fateps[L]\ee^{\fatQ^\top s}\fat{e}_j\rmd s.
	\end{equation}
\end{theorem}
\begin{proof}
Recall that $\fatone^\top\EE[\fat{\hat V}_t]=\EE[V_t]$ for $\fat{\hat{V}}_t:=V_t\fatLambda_t$. Using \eqref{integrationbyparts} and the SDE \eqref{MMGOUSDE}, and observing that due to Lemma \ref{lem-supmomentV} we can apply Lemma \ref{eq:unif.mart.int.M}, we obtain for any $j\in S$
	$$	\EE_j\big[\fat{\hat{V}}_t\big]
			= \EE_j \big[\fat{\hat{V}}_0\big] + \fatQ^\top \EE_j \bigg[\int_{(0,t]} \fat{\hat{V}}_{s-} \rmd s \bigg] 
			+ \EE_j \bigg[ \int_{(0,t]} \fat{\hat{V}}_{s-} \rmd U_s \bigg] + \EE_j \bigg[ \int_{(0,t]} \fatLambda_{s-} \rmd L_s \bigg] + \EE_j\Big[ [V,\fatLambda]_t\Big].$$
	Further, using  a Fubini argument and \eqref{eq:Theorem_Mean_Time_MAPintegral_vectorvalued}, this implies
	\begin{align}
		\EE_j\big[\fat{\hat{V}}_t\big]
		&= 	\EE_j\big[\fat{\hat{V}}_0\big]
		+\fatQ^\top\int_{(0,t]} \EE_j\big[\fat{\hat{V}}_s\big]\rmd s
		+\fateps[U]\int_{(0,t]} \EE_j\big[\fat{\hat{V}}_s\big]\rmd s - \EE_j \Big[ \int_{(0,t]} V_{s-} \rmd [U,\fatLambda]_s \Big]  \nonumber \\
		&\quad  
		+\fateps[L]\int_{(0,t]} \EE_j\left[\fatLambda_s\right]\rmd s - \EE_j \Big[ \int_{(0,t]}  \rmd [L,\fatLambda]_s \Big] + \EE_j\Big[[V,\fatLambda]_t\Big] \nonumber \\
		&= 	\EE_j\big[\fat{\hat{V}}_0\big]
		+\fatQ^\top\int_{(0,t]} \EE_j\big[\fat{\hat{V}}_s\big]\rmd s
		+\fateps[U]\int_{(0,t]} \EE_j\big[\fat{\hat{V}}_s\big]\rmd s  
		+\fateps[L]\int_{(0,t]} \EE_j\left[\fatLambda_s\right]\rmd s ,\label{eq-meanhelp1}
	\end{align} 
	since
	$$\int_{(0,t]} V_{s-} \rmd [U,\fatLambda]_s + \int_{(0,t]}  \rmd [L,\fatLambda]_s  = 
	\Big[ \int_{(0,\cdot]} V_{s-} \rmd U_s + L, \fatLambda \Big]_t = [V,\fatLambda]_t.$$
	Inserting \eqref{eq-expfatLambda} yields the ODE 
	\begin{align*}
		\frac{\rmd}{\rmd t}\EE_j\big[\fat{\hat{V}}_t\big]=\left(\fatQ^\top+\fateps[U]\right)\EE_j\big[\fat{\hat{V}}_t\big]+\fateps[L] \ee^{\fatQ^\top t}\fat{e}_j,
	\end{align*}
	which is solved uniquely by
	\begin{align*}
		\EE_j\big[\fat{\hat{V}}_t\big]=
		\ee^{(\fatQ^\top+\fateps[U])t} \EE_j\big[\fat{\hat{V}}_0\big]
		+\int_0^t\ee^{(\fatQ^\top
			+\fateps[U])(t-s)}\fateps[L] \ee^{\fatQ^\top s}\fat{e}_j\rmd s.
	\end{align*}
	Multiplying this expression by $\fatone^\top$ yields \eqref{eq-GOUrunningmean}. 
\end{proof}

The next theorem proves that the autocovariance function of the MMGOU decreases exponentially if the leading eigenvalue of the matrix exponent $\fat{\Psi}_\xi(-1)$ is negative. 

\begin{theorem}\label{MMGOU_Theorem_Autocorrelation}
	Consider the MMGOU process $(V_t)_{t\geq 0}$ driven by the bivariate MAP $((\xi,\eta),J)$ and solving the SDE \eqref{MMGOUSDE} for the bivariate MAP $((U,L),J)$ in \eqref{eq-ULviaxieta}. Assume that for all $j\in S$
	$$\EE_j\Big[\sup_{0<s\leq 1}\ee^{2|\xi_s|}\Big]<\infty,\quad  \EE_j\big[|V_0|^{2}\big]<\infty,\quad \text{and }\; \EE_j\big[|\eta_1|^{2}\big]<\infty.$$
	 Then the autocovariance structure of $(V_t)_{t\geq 0}$ can be described via 
	\begin{align}\label{eq-acf}
		\begin{pmatrix}
			\cov_j(\fatLambda_{t},V_s)\\
			\cov_j(\hatfat{V}_{t},V_s)
		\end{pmatrix}
		=\ee^{\fat{K}(t-s)}
		\begin{pmatrix}
			\cov_j(\fatLambda_{s},V_s)\\
			\cov_j(\hatfat{V}_{s},V_s)
		\end{pmatrix},\quad\text{with} \quad  \fat{K}:=\begin{pmatrix}
			\fatQ^\top&0\\
			\fateps[L]&\fatQ^\top+\fateps[U]
		\end{pmatrix}.
	\end{align}
	 for all $0\leq s\leq t$ with starting values $\cov_j(\fatLambda_{0},V_0)=0$ and $\cov_j(\hatfat{V}_{0},V_0)= \fat{e}_j \var(V_0)$. \\
	  In particular, $\cov_j(V_t,V_s) = \fatone^\top\cov_j(\hatfat{V}_t,V_s)$ decreases exponentially if $\lambda_{\max}^\xi(-1)<0$. 
\end{theorem}
	\begin{proof}
	We start to consider $\cov_j(\hatfat{V}_t,V_s) = \EE_j[\hatfat{V}_t V_s]- \EE_j[\hatfat{V}_t] \EE_j[V_s]$ and derive via partial integration over $(s,t]$ 
		\begin{align}\label{eq-acfhelp1}
			\cov_j\big(\hatfat{V}_t,V_s\big) &= \cov_j\big(\hatfat{V}_s,V_s\big)
			+\cov_j\bigg(\int_{(s,t]} V_{u-} \rmd\fatLambda_u, V_s\bigg)
			+\cov_j\bigg(\int_{(s,t]} \fatLambda_{u-} \rmd V_u ,V_s\bigg)\\
&			+\cov_j\left(\big[\fatLambda, V\big]_t- \big[\fatLambda, V\big]_s, V_s\right).\nonumber 
		\end{align}
	Hereby, via \eqref{MMGOUSDE} we have
	\begin{align*}
		\cov_j\bigg(\int_{(s,t]} \fatLambda_{u-} \rmd V_u ,V_s\bigg) &= \cov_j\bigg(\int_{(s,t]} \hatfat{V}_{u-} \rmd U_u ,V_s\bigg) + \cov_j\bigg(\int_{(s,t]} \fatLambda_{u-} \rmd L_u ,V_s\bigg), \quad \text{and} \\
		\cov_j\left(\big[\fatLambda, V\big]_t- \big[\fatLambda, V\big]_s, V_s\right) 
		&= \cov_j\bigg(\int_{(s,t]} V_{u-} \rmd \big[\fatLambda, U\big]_u, V_s\bigg)+ \cov_j\bigg(\int_{(s,t]} \rmd \big[\fatLambda, L\big]_u, V_s\bigg),
	\end{align*}
	while due to \eqref{eq-lambdadecompose} 
	\begin{align*}
		\cov_j\bigg(\int_{(s,t]} V_{u-} \rmd\fatLambda_u, V_s\bigg) &= \cov_j\bigg(\fatQ^\top \int_{(s,t]} \hatfat{V}_{u} \rmd u, V_s\bigg)+ \cov_j\bigg(\int_{(s,t]} V_{u-} \rmd \fat{M}_u, V_s\bigg).
	\end{align*}
As $\EE_j[\sup_{0<s\leq t} |V_s V_t|]<\infty$ by Lemma \ref{lem-supmomentV}, Lemma \ref{eq:unif.mart.int.M} yields that
\begin{align*}
\cov_j\bigg(\int_{(s,t]} V_{u-} \rmd \fat{M}_u, V_s\bigg)&= \EE_j \bigg[\int_{(s,t]} V_s V_{u-} \rmd \fat{M}_u \bigg] - \EE_j \left[V_s\right] \EE_j\bigg[\int_{(s,t]} V_{u-} \rmd \fat{M}_u \bigg] = \mathbf{0}.
\end{align*}
Inserting the above in \eqref{eq-acfhelp1} and applying \eqref{eq:Theorem_Mean_Time_MAPintegral_vectorvalued} we obtain by a straightforward computation 
\begin{align}
	\lefteqn{\cov_j\big(\hatfat{V}_t,V_s\big)} \nonumber \\ 
		&= \cov_j\big(\hatfat{V}_s,V_s\big) + \big(\fatQ^\top + \fateps[U]\big) \int_{(s,t]} \cov_j\big(\hatfat{V}_u, V_s \big) \rmd u + \fateps[L] \int_{(s,t]} \cov_j\big(\fatLambda_u, V_s \big) \rmd u. \label{eq-acfhelp2}
	\end{align}
 Moreover, again using \eqref{eq-lambdadecompose} and Lemma \ref{eq:unif.mart.int.M}, we derive
\begin{align}\cov_j\big(\fatLambda_t, V_s \big) &=\cov_j\big(\fatLambda_s, V_s \big)+  \cov_j\bigg( \fatQ^\top \int_{(s,t]} \fatLambda_u \rmd u, V_s \bigg) + \cov_j\bigg(\int_{(s,t]}\rmd \fat{M}_u, V_s \bigg) \nonumber \\
	&= \cov_j\big(\fatLambda_s, V_s \big) +  \fatQ^\top \int_{(s,t]} \cov_j\Big(\fatLambda_u, V_s \Big) \rmd u. \label{eq-acfhelp3} \end{align}
Equations \eqref{eq-acfhelp2} and \eqref{eq-acfhelp3} together can now be reformulated in matrix and differential form as
$$	\frac{\rmd}{\rmd t}\begin{pmatrix}
	\cov(\fatLambda_t,V_s)\\ 	\cov(\hatfat{V}_t,V_s)
\end{pmatrix} = \begin{pmatrix}
\fatQ^\top&0\\
\fateps[L]&\fatQ^\top+\fateps[U]
\end{pmatrix} \begin{pmatrix}
\cov(\fatLambda_s,V_s)\\ 	\cov(\hatfat{V}_s,V_s)
\end{pmatrix}=: \fat{K} \begin{pmatrix}
\cov(\fatLambda_s,V_s)\\ 	\cov(\hatfat{V}_s,V_s)
\end{pmatrix},\; 0\leq s\leq t,$$
with unique solution as given in \eqref{eq-acf}. \\
The autocovariance function is thus decreasing exponentially if and only if all eigenvalues of $\fat{K}$ have non-positive real part. Denoting $\fat{I}=\diag(1)$, we see that
$$\det (\fat{K} - \lambda \fat{I}) = \det(\fatQ^\top -\lambda \fat{I}) \det(\fatQ^\top + \fateps[U]-\lambda \fat{I}).$$
As, in our setting, all eigenvalues of the intensity matrix $\fat{Q}$ have non-positive real part, cf. \cite{Norris}, it is sufficient to consider the eigenvalues of $\fatQ^\top + \fateps[U],$ i.e. the eigenvalues of $\fat\Psi_\xi(-1)$ by \eqref{eq_invertiblematrices}. This finishes the proof. \end{proof}

\subsection{Stationary MMGOU processes}\label{S4b}

As already mentioned in the introduction, under our standing assumption that  $J$ is defined on a finite state space $S$ and is ergodic with stationary distribution $\pi$, the MMGOU process \eqref{MMGOUexplicit} admits a non-trivial stationary distribution if and only if the integral $\int_{(0,t]} \ee^{\xi^\ast_{s-}} \rmd L^*_s$ converges {\pa $\PP_\pi^*$-almost} surely as $t\rightarrow \infty$ to some finite-valued random variable $V_\infty$. In this case the stationary distribution of the MMGOU process under $\PP_\pi$ is uniquely determined as the distribution of $V_\infty$  given in \eqref{eq_defVinfty}. In this section we analyse the moments of $V_\infty$. We start by providing conditions for their existence before presenting explicit formulas for some integer moments. 

\begin{theorem}\label{thm_condstat}
	Consider the MMGOU process $(V_t)_{t\geq 0}$ driven by the bivariate MAP $((\xi,\eta),J)$ and solving the SDE \eqref{MMGOUSDE} for the bivariate MAP $((U,L),J)$ in \eqref{eq-ULviaxieta}. Assume that $\lim_{t\to\infty}\xi_t=\infty$ $\PP_\pi$-a.s. and that there exists  $\kappa\geq 1$ such that for all $j\in S$
	\begin{equation}\label{eq_condstatmoment} \EE_j^\ast\bigg[\bigg|\int_{(0,\taure_1(j)]}\ee^{\xi^*_{s-}}\rmd L^*_s\bigg|^\kappa\bigg]<\infty \quad \text{and} \quad \EE_j^\ast\Big[\ee^{\kappa\xi^*_{\taure_1(j)*}}\Big]<1.\end{equation}
	Then $(V_t)_{t\geq 0}$ has a stationary solution with distribution $V_t\overset{d}=V_\infty$, $t\geq 0$, and in particular $\EE_\pi[|V_\infty|^\kappa]<\infty$.
\end{theorem}
\begin{proof}
	As shown in \cite[Thm. 3.3]{BEHME+SIDERIS_MMGOU}, there exists a finite random variable $V_\infty$ such that $(V_t)_{t\geq 0}$ started with $V_0\overset{d}=V_\infty$ is strictly stationary, if the exponential functional $\int_{(0,t]}\ee^{\xi^*_{s-}}\rmd L^*_s$ converges in $\PP_\pi^\ast$-probability as $t\to \infty$. This however follows via \cite[Prop's 5.2 and 5.7]{BEHME+SIDERIS_ExpFuncMAP2020} from our assumptions $\lim_{t\to\infty}\xi_t=\infty$ $\PP_\pi$-a.s. and $\EE_j[|\int_{(0,\taure_1(j)]}\ee^{\xi^*_{s-}}\rmd L^*_s|^\kappa]<\infty$, since the latter implies $\EE_j[\log^+(\int_{(0,\taure_1(j)]}\ee^{\xi^*_{s-}}\rmd L^*_s)]<\infty$.\\
	It thus remains to prove the finiteness of the $\kappa$th moment of $V_\infty$.\\ 
	By \cite[Thm. 3.3]{BEHME+SIDERIS_MMGOU} and \eqref{def_stationarydistirbution_Markov_} it holds
	\begin{align*}
		\PP_\pi^*\left(V_\infty\in \cdot\right)&=\PP_\pi^*\bigg( -\int_{(0,\infty)} \ee^{\xi_{s-}^*}\rmd L_s^*\in\cdot\bigg)=\sum_{j\in S}\pi^*_j\cdot\PP_j^*\bigg( -\int_{(0,\infty)} \ee^{\xi_{s-}^*}\rmd L_s^*\in\cdot\bigg).
	\end{align*}
	Fix $j\in S$ and let $N^{re,*}_t$ denote the number of returns of $J^*$ to state $j$ up to time $t>0$. Then under $\PP_j^\ast$ for any $t\geq 0$ we can rewrite the exponential integral as 
	\begin{align*}
		-\int_{(0,t]} \ee^{\xi^*_{s-}}\rmd L_s^*&=-\int_{(0,\taure_{N^{re,*}_t}(j)^*]} \ee^{\xi^*_{s-}}\rmd L_s^*-\ee^{\xi^*_{\taure_{N^{re,*}_t}(j)^*}}\int_{(\taure_{N^{re,*}_t}(j)^*,t]} \ee^{\xi^*_{s-}-\xi^*_{\taure_{N^{re,*}_t}(j)^*}}\rmd L_s^*.
	\end{align*}
	Using the same arguments as in the proof of \cite[Prop. 4.12]{BEHME+SIDERIS_ExpFuncMAP2020}, the last factor $\int_{(\taure_{N^{re,*}_t}(j)^*,t]} \exp(\xi^*_{s-}-\xi^*_{\taure_{N^{re,*}_t}(j)^*})\rmd L_s^*$ converges in distribution as $t \to\infty$, {\pa while $\ee^{\xi^*_{\taure_{N^{re,*}_t}(j)^*}}$ converges to zero a.s.\ by assumption. Hence, Slutsky's theorem yields} 
	\begin{equation} \label{eq_constatmomenthelp}
		\lim_{t\to\infty}\bigg(-\int_{(0,t]} \ee^{\xi^*_{s-}}\rmd L_s^* \bigg) \iv \lim_{t\to\infty}\bigg(-\int_{(0,\taure_{N^{re,*}_t}(j)^*]}\ee^{\xi^*_{s-}}\rmd L_s^*\bigg)
	\end{equation}
	under $\PP_j^*$. Thus, also 
	\begin{align*}\EE_j^\ast\bigg[\bigg|	\lim_{t\to\infty}\bigg(-\int_{(0,t]} \ee^{\xi^*_{s-}}\rmd L_s^* \bigg) \bigg|^\kappa\bigg]& = \EE_j^\ast\bigg[\bigg|\lim_{t\to\infty}\bigg(-\int_{(0,\taure_{N^{re,*}_t}(j)^*]}\ee^{\xi^*_{s-}}\rmd L_s^*\bigg)\bigg|^\kappa\bigg]\\ &=\EE_j^\ast\bigg[\bigg|\lim_{n\to\infty}\bigg(\int_{(0,\taure_n(j)^*]}\ee^{\xi^*_{s-}}\rmd L_s^*\bigg)\bigg|^\kappa\bigg].\end{align*}
	It remains to show finiteness of the latter to prove the statement. To this aim we define
	\begin{align*}
		B_m^{*,j}:=\int_{(\taure_{m-1}(j)^*,\taure_m(j)^*]}\ee^{\xi^*_{s-}-\xi^*_{\taure_{m-1}(j)^*}}\rmd L_s^*, \quad m\in\NN.
	\end{align*}
	Then under $\PP^\ast_j$ the sequence $(B_m^{*,j})_{m\geq 1}$ is i.i.d. with $\EE_j^\ast [|B_1^{\ast,j}|^\kappa]<\infty$ and for every $m\geq 1$ the variable $B_m^{\ast,j}$ is independent of $\xi^*_{\taure_{m-1}(j)^*}$. We can thus follow the lines of the proof of  \cite[Prop. 4.1]{LINDNER+MALLER_LevyintegralsStationarityGOUP} 
	and derive via Hölder's inequality that
	\begin{align} \EE_j^*\bigg[\bigg|\int_{(0,\taure_n(j)^*]}\ee^{\xi^*_{s-}}\rmd L_s^*\bigg|^\kappa\bigg]
		 &\leq 
	\EE_j^*\left[|B^{*,j}_1|^\kappa\right]\bigg(\sum_{\ell=1}^n \EE_j^*\Big[\ee^{\kappa\xi^*_{\taure_{\ell-1}(j)^*}}\Big]^\frac{1}{\kappa}\bigg)^{\lfloor\kappa\rfloor}\cdot\bigg(\sum_{\ell=1}^n\EE_j^*\Big[\ee^{\kappa \xi_{\taure_{\ell-1}(j)^*}^*}\Big]^{\frac{\kappa - \lfloor \kappa\rfloor}{\kappa}}\bigg), \label{eq-LMhoelder}
	 \end{align}
	where for $\kappa=\lfloor \kappa \rfloor$ the second factor can be omitted. 
	To show the absolute convergence of the two sums, note that for any $\ell\geq 1$ it holds
	\begin{align*}
		\xi_{\taure_\ell(j)^*}^*=\sum_{m=1}^\ell \big( \xi^*_{\taure_m(j)^*}-\xi^*_{\taure_{m-1}(j)^*}\big),
	\end{align*}
	which is a sum of i.i.d. random variables $(\xi^*_{\taure_m(j)^*}-\xi^*_{\taure_{m-1}(j)^*})_{m\geq 1}$  with  $\xi^*_{\taure_m(j)^*}-\xi^*_{\taure_{m-1}(j)^*}\iv \xi^*_{\taure_1(j)^*}$. Hence 
		\begin{align*}
		\EE_j^*\left[\ee^{\kappa\xi^*_{\taure_\ell(j)^*}}\right]=\EE_j^*\left[\ee^{\kappa\xi^*_{\taure_1(j)^*}}\right]^\ell, \quad \ell\geq 1,
	\end{align*}
	and inserting this in \eqref{eq-LMhoelder} we get
	\begin{align*}
		\EE_j^*\bigg[\bigg|\int_{(0,\taure_n(j)^*]}\ee^{\xi^*_{s-}}\rmd L_s^*\bigg|^\kappa\bigg]& \leq 
		\EE_j^*\left[|B^{*,j}_1|^\kappa\right]\Bigg(\sum_{\ell=1}^n\EE_j^*\Big[\ee^{\kappa\xi^*_{\taure_{1}(j)^*}}\Big]^\frac{\ell}{\kappa}\Bigg)^{\lfloor\kappa\rfloor}\cdot\Bigg(\sum_{\ell=1}^n\EE_j^*\Big[\ee^{\kappa \xi_{\taure_{1}(j)^*}^*}\Big]^{\frac{\ell(\kappa - \lfloor \kappa\rfloor)}{\kappa}}\Bigg). 
	\end{align*}
	Passing to the limit as $n\to \infty$ this converges absolutely due to assumption \eqref{eq_condstatmoment}. Since we already proved the existence of the limit in \eqref{eq_constatmomenthelp} this implies the statement.
\end{proof}

\begin{cor}	\label{cor_condstat} Consider the MMGOU process $(V_t)_{t\geq 0}$ driven by the bivariate MAP $((\xi,\eta),J)$ and solving the SDE \eqref{MMGOUSDE} for the bivariate MAP $((U,L),J)$ in \eqref{eq-ULviaxieta}. Assume that there exists  $\kappa\geq 1$ such that
\begin{enumerate}
	\item  $\psi_{\xi^{(j)}}(-\kappa)<|q_{jj}|$ for all $j\in S$,  and  $\EE[\exp(-\kappa |Z^{ij}_{\xi,1}|)]<\infty$ for all $(i,j)\in S^2$ with $q_{ij}\neq 0$, 
	\item 	$$ \max_{i\in S, i\neq j}\bigg( \frac{1}{|q_{ii}|-\psi_{\xi^{(i)}}(-\kappa)} \sum_{\substack{\ell \in S\\ \ell\neq i,j}} q_{i\ell} \EE\Big[\ee^{-\kappa Z_{\xi,1}^{i\ell}} \Big]\bigg)<1, \quad j\in S,$$
		\item the leading eigenvalue $\lambda_{\max}^\xi(-\kappa)$ of $\fat\Psi_\xi(-\kappa)$ is negative, i.e. $\fat\Psi_\xi(-\kappa)$ is negative definite,			
		\item $\EE[|L_1^{(j)}|^{\kappa}]<\infty$ for all $j\in S$ and $\EE[|Z^{ij}_{L,1}|^{\kappa}]<\infty$ for all $(i,j)\in S^2$ such that $q_{ij}>0$.
	\end{enumerate}
Then $(V_t)_{t\geq 0}$ has a stationary solution with distribution $V_t\overset{d}=V_\infty$, $t\geq 0$, and in particular $\EE_\pi[|V_\infty|^\kappa]<\infty$.
\end{cor}
\begin{proof} We will check the conditions of Theorem \ref{thm_condstat} to prove the statement.\\ 
First of all, since the function $w\mapsto\lambda_{\max}^\xi(w)$ is strictly convex and smooth, cf. \cite[Prop. 3.4]{KKPW_HittingTimeofZero}, and since $\lambda_{\max}^\xi (0)=0$, assumption (iii) implies that $(\lambda_{\max}^\xi)'(0)>0$. As moreover $(\lambda_{\max}^\xi)'(0+)=\EE_\pi[\xi_1]$, we conclude that $\lim\xi_t=\infty$ a.s. as $t\to\infty$, cf. \cite[Cor. XI.2.9]{ASMUSSEN_AppliedProbandQueues}.\\
To prove the conditions in \eqref{eq_condstatmoment}, observe first that due to assumptions (i) and (ii) 
by Proposition \ref{detPsiEEexplessthan1}(i) we have $\EE_j[\ee^{\kappa \xi_{\taure_1(j)^\ast}^\ast}]<\infty$ for all $j\in S$. Further  note that relation \eqref{eq:lap.tran.dual} entails that if $\fat\Psi_{\xi}(-\kappa)$ is negative definite the same holds true for $\fat\Psi_{\xi^*}(\kappa)$. Thus by Proposition \ref{detPsiEEexplessthan1}(ii) we even have $\EE_j^\ast [\ee^{\kappa\xi^*_{\taure(j)*}}]<1$ for all $j\in S$.\\
 Lastly, $\EE_j^\ast [|\int_{(0,\taure_1(j)^*]}\ee^{\xi_{s-}^*}\rmd L^*_s|^\kappa]<\infty$ follows by Theorem \ref{Theo_mean_MAP_integral}, if we can guarantee that $\EE_j^\ast[|L_1^\ast|^{\kappa}]<\infty$ and $\EE_j^\ast[\sup_{0<s\leq \taure_1(j)^\ast} \ee^{\kappa \xi_{s-}^\ast}]<\infty.$ Hereby, $\EE_j^\ast[|L_1^\ast|^{\kappa}]<\infty$ follows from assumption (iv) via Lemma \ref{prop:mom.map.lev.ju}, since $L_t^{\ast,(j)} \overset{d}= - L_t^{(j)}$ and $Z_{L^\ast,1}^{ij} \overset{d}= -Z_{L,1}^{ji}$.  Moreover, $\EE_j^\ast[\sup_{0<s\leq \taure_1(j)^\ast} \ee^{\kappa \xi_{s-}^\ast}]<\infty$ is equivalent to  $\EE_j^\ast[\ee^{\kappa \xi_{\taure_1(j)^\ast}^\ast}]<\infty$ due to  \cite[Thm. 25.18]{SATO_LPinfinitelydivisibledistributions} and a similar computation as in \eqref{eq-momentexistencehelp} conditioning on $\taure_1(j)^\ast$.
\end{proof}

Under the above derived conditions for the existence of the $\kappa$'th moment of the stationary distribution, the following theorem provides explicit formulae for the first and second moment.

\begin{theorem}\label{MMGOUTheoremMoments} Consider the MMGOU process $(V_t)_{t\geq 0}$ driven by the bivariate MAP $((\xi,\eta),J)$ and solving the SDE \eqref{MMGOUSDE} for the bivariate MAP $((U,L),J)$ in \eqref{eq-ULviaxieta}. Assume the conditions of Corollary \ref{cor_condstat} are fulfilled for $\kappa=1$, or $\kappa=2$, respectively.  Then the first two moments of the stationary distribution $V_\infty$ of the MMGOU process are given by 
 \begin{align*}
\EE_\pi\left[V_\infty \right]&=-\fatone^\top\cdot\big(\fatQ^\top+\fateps[U]\big)^{-1}\cdot\fateps[L]\cdot\fatpi, \quad \text{and}\\
 \EE_\pi\left[V_\infty^2\right]&=  \fat{1}^\top \big(\fatQ^\top + 2 \fateps[U] + \fateps\big[[U,U]\big] \big)^{-1} \cdot \\ & \qquad \cdot\left(2 \left(  \fateps[L] + \fateps\big[[U,L]\big]\right) \big(\fatQ^\top + \fateps[U] \big)^{-1} \fateps[L]   - \fateps\big[[L,L]\big] \right)\fatpi.
\end{align*}
\end{theorem}

\begin{proof}
	To compute the mean, we follow the lines of the proof of Theorem \ref{Prop_runningmean}, where we note that all needed conditions are met due to Lemma \ref{prop:mom.map.lev.ju} and Lemma \ref{Relation(U,xi)integrability}, due to the relation \eqref{eq-ULviaxieta}, and due to Corollary \ref{cor_condstat}. In particular, by Lemma \ref{lem-supmomentV}, it holds $\EE_\pi[\sup_{0<s\leq t} |V_s|]<\infty.$ As we are in the stationary regime here, we clearly have that $\EE_\pi\big[\hatfat{V}_t\big]=\EE_\pi\big[\hatfat{V}_0\big]$ and $\EE_\pi\left[\fatLambda_t\right]=\fatpi$ for all $t\geq 0$. Thus \eqref{eq-meanhelp1} simplifies to 
	\begin{align*}
	\mathbf{0}&= \fatQ^\top\int_{(0,t]} \EE_\pi\big[\fat{\hat{V}}_s\big]\rmd s
		+\fateps[U]\int_{(0,t]} \EE_\pi \big[\fat{\hat{V}}_s\big]\rmd s  
		+\fateps[L]\int_{(0,t]} \EE_\pi \left[\fatLambda_s\right]\rmd s\\
\Rightarrow\quad	\mathbf{0}	&= \big( \fatQ^\top  		+\fateps[U]\big) \EE_\pi \big[\fat{\hat{V}}_0 \big]
		+\fateps[L] \cdot  \fatpi.
	\end{align*}
This immediately implies
\begin{equation} \label{eq-meanstatMMGOU} \EE_\pi[V_\infty] = \fatone^\top \EE_\pi \big[\fat{\hat{V}}_0 \big] = - \fatone^\top \big( \fatQ^\top  		+\fateps[U]\big)^{-1} \cdot\fateps[L] \cdot  \fatpi \end{equation}
as stated, where we note that by \eqref{eq_invertiblematrices} $\fatQ^\top  +\fateps[U]=\fat{\Psi}_\xi(-1)$, which is invertible as by assumption $\lambda_{\max}^\xi(-1)<0$. \\
For the second moment, following the first lines of the proof of Theorem \ref{Thm-MAPvariance} we derive
\begin{align}
	\EE_\pi\big [V_t^2 \fatLambda_t\big] &= \EE_\pi[V_0^2 \fatLambda_0]+ \fatQ^\top \int_{(0,t]} \EE_j[V_s^2\fatLambda_s] \rmd s + 2 \EE_j \bigg[\int_{(0,t]} \hatfat{V}_{s-} \rmd V_s \bigg]\nonumber  \\ & \quad + \EE_j\bigg[\int_{(0,t]} \fatLambda_{s-} \rmd [V,V]_s\bigg] + \EE_j\Big[[V^2,\fatLambda]_t\Big].\label{eq-secondmomstatMMGOU}
\end{align}
Hereby notice that Lemma \ref{eq:unif.mart.int.M} was applicable since $\EE_j[\sup_{0<s\leq t} |V_s|^2]<\infty$ follows from our assumptions due to Lemma \ref{lem-supmomentV}. \\
By stationarity it holds $\EE_\pi [V_t^2 \fatLambda_t] = 	\EE_\pi[V_0^2 \fatLambda_0]$. Inserting this, the SDE \eqref{MMGOUSDE} and 
\begin{align*}
	[V,V]_t 	&= \int_{(0,t]} V_{s-}^2 \rmd [U,U]_t + 2 \int_{(0,t]} V_{s-} \rmd[U,L]_s + [L,L]_t, \quad t\geq 0,
\end{align*}
in \eqref{eq-secondmomstatMMGOU} then yields by a straightforward computation using \eqref{eq:Theorem_Mean_Time_MAPintegral_vectorvalued}
\begin{align*}
\fat{0}
	&= \fatQ^\top \int_{(0,t]} \EE_\pi\big[V_{s}^2  \fatLambda_{s}\big] \rmd s + 2 \fateps[U] \int_{(0,t]} \EE_\pi\big[ V_{s}^2 \fatLambda_{s} \big] \rmd s + 2 \fateps[L] \int_{(0,t]} \EE_\pi\big[\hatfat{V}_{s} \big] \rmd s\\
	&\quad + \fateps\big[[U,U]\big] \int_{(0,t]} \EE_\pi\big[ V_{s}^2 \fatLambda_{s} \big] \rmd s + 2  \fateps\big[[U,L]\big]\int_{(0,t]} \EE_\pi\big[\hatfat{V}_{s} \big] \rmd s + \fateps\big[[L,L]\big] \int_{(0,t]} \EE_\pi\big[\fatLambda_s \big] \rmd s\\
	&= \big(\fatQ^\top + 2 \fateps[U] + \fateps\big[[U,U]\big] \big)  \EE_\pi\big[V_{0}^2  \fatLambda_{0}\big] t + 2 \left(  \fateps[L] + \fateps\big[[U,L]\big]\right) \EE_\pi\big[\hatfat{V}_{0} \big]t + \fateps\big[[L,L]\big] \fatpi t.
\end{align*}
Inserting the formula \eqref{eq-meanstatMMGOU} for the first moment this proves
$$\big(\fatQ^\top + 2 \fateps[U] + \fateps\big[[U,U]\big] \big)  \EE_\pi\big[V_{0}^2  \fatLambda_{0}\big] =  2 \left(  \fateps[L] + \fateps\big[[U,L]\big]\right) \big(\fatQ^\top + \fateps[U] \big)^{-1} \fateps[L]  \fatpi - \fateps\big[[L,L]\big] \fatpi. $$
This easily yields the given formula for $\EE_\pi[V_\infty^2]=\EE_\pi[V_0^2] = \fat{1}^\top \EE_\pi\big[V_{0}^2  \fatLambda_{0}\big]$. Indeed, the identity $\fatQ^\top  		+2\fateps[U] + \fateps\big[[U,U]\big]=\fat{\Psi}_\xi(-2)$ holds by \eqref{eq_invertiblematrices}, and since $\lambda_{\max}^\xi(-2)<0$, this matrix is invertible.
\end{proof}

\begin{remark}
	With a similar procedure as used in the proof of Theorem \ref{MMGOUTheoremMoments} and using ideas as in the proof of Proposition \ref{expectationmatrixstochasticexponential} one can derive a recursion formula for higher moments of the stationary distribution, given they exist. 
For $k>2$ this yields the recursion
\begin{align*}
\lefteqn{\EE_\pi[V_\infty^k]= \mathbf{1}^\top \EE_\pi\big[V_0^k\fatLambda_0\big] }\\
&=  -\mathbf{1}^\top\bigg( \fatQ^\top + k \fateps[U^c] + \tfrac{k(k-1)}{2} \fateps[\langle
	U^c,U^c \rangle ] + \fateps\Big[\sum_{0<s\leq \cdot}\left((1+\Delta U_s)^k -1\right)\Big] \bigg)^{-1} \cdot\\
	&\quad\cdot \Bigg[ \bigg(k \fateps[L^c] + k(k-1) \fateps[\langle U^c, L^c\rangle] + \fateps\Big[\sum_{0<s\leq \cdot} k \Delta L_s  (1+\Delta U_s)^{k-1}  \Big] \bigg)\EE_\pi\big[V_\infty^{k-1}\fatLambda_0\big] \\ 
	&\qquad + \Big(\tfrac{k(k-1)}{2} \fateps[\langle
	L^c,L^c \rangle]  \Big) \EE_\pi\big[V_\infty^{k-2}\fatLambda_0\big] 
	+ \sum_{n=2}^{k} \fateps\Big[\sum_{0<s\leq \cdot} {k\choose n} (1+ \Delta U_s)^{k- n} \Delta L_s^n \Big]\EE_\pi\big[V_\infty^{k-n}\fatLambda_0\big]\Bigg],
\end{align*}
assuming that all needed moment conditions are fulfilled.
Note that by Proposition \ref{expectationmatrixstochasticexponential} the inverted matrix in the recursion equals $\fat{\Psi}_\xi(-k)$ and hence the inverse exists, whenever $\lambda_{\max}^\xi(-k)<0$. Further, we observe that for continuous processes $U$ and $L$ the recursion stops with the $k-2$'nd moment. This agrees with the recursion for the moments of the MMOU process derived in \cite[Sec. 3.4]{HUANG+SPREJ_MarkovmodulatedOUP2014}.

\end{remark}

\anita
\section{Additional proofs for the results in Section \ref{S3}} \label{Sproofs}

\subsection{Proofs for the results in Section \ref{S3a}}

\normal

\begin{proof}[Proof of Lemma \ref{prop:mom.map.lev.ju}] Throughout the proof we set $t=1$. The proof for general $t>0$ works alike.\\
	Obviously (i) implies (ii), while the equivalence of (ii) and (iii) is immediate since $\EE_\pi[|X_1|^\kappa] = \sum_{j=1}^{|S|} \pi_j \EE_j[|X_1|^\kappa]$.
	To prove the remaining conclusions we follow ideas from the proof of \cite[Thm. 34]{DEREICH+DOERING+KYPRIANOU_RealselfsimilarprocessesStartedfromOrigin2017}: Assume first (iii). Fix any $j\in S$ and consider the event $O$, that $J$ has no transition in $[0,1]$. Then
	\begin{align*}
		\infty>\EE_j\left[|X_1|^\kappa\right]=\PP_j(O)\cdot\EE_j\left[|X_1|^\kappa|O\right]+(1-\PP_j(O))\EE_j\left[|X_1|^\kappa|O^c\right]\geq \ee^{-|q_{jj}|}\EE\big[ |X_1^{(j)}|^\kappa\big],
	\end{align*}
	implying that $\EE\big[ |X_1^{(j)}|^\kappa\big]<\infty$ for all $j\in S$.
	Now consider the event $I$, that the first transition of $J$ happens before $t=1$ and the second after $t=1$. Then
	\begin{align*}
		\infty>\EE_j\left[|X_1|^\kappa\right]& \geq  \PP_j(I)\cdot\EE_j\left[|X_1|^\kappa|I\right] \\ 
		&= \int_0^1 |q_{jj}|\ee^{-|q_{jj}|s} \sum_{i\neq j}\frac{q_{ji}}{|q_{jj}|} \ee^{-|q_{ii}|(1-s)}\EE_j\left[|X_s^{(j)}+Z_{X,1}^{ji}+X_{1-s}^{(i)}|^\kappa\right]\rmd s.
	\end{align*}
	Hence, especially, $\EE_j\big[|X_s^{(j)}+Z_{X,1}^{ji}+X_{1-s}^{(i)}|^\kappa\big]<\infty$ Lebesgue a.e. in $[0,1]$ for {\pa every $(i,j)\in S\p2$}, and we can conclude with \eqref{jensenestimate} that for any such $s\in[0,1]$
	\begin{align*}
		\EE_j\big[|Z_{X,1}^{ji}|^\kappa\big]
		& = \EE_j  \left[|(X_s^{(j)}+Z_{X,1}^{ji}+X_{1-s}^{(i)})+(-X_s^{(j)}) +(-X_{1-s}^{(j)})|^\kappa\right]\\
		& \leq
		3^{\kappa-1}  \left(\EE_j\left[|(X_s^{(j)}+Z_{X,1}^{ji}+X_{1-s}^{(i)})|^\kappa\right]+\EE_j\big[|X_s^{(j)}|^\kappa\big]+\EE_i\big[|X_{1-s}^{(i)}|^\kappa\big]\right)<\infty.
	\end{align*}
	This implies (iv).\\
	\pa Finally, we show that (iv) implies (i). So, assume (iv) and note that by \cite[Thm 25.18]{SATO_LPinfinitelydivisibledistributions} our assumptions imply that  $\EE\left[\sup_{0<s\leq 1}|X_s^{(j)}|^\kappa\right]<\infty$ for all $j\in S$.  By \eqref{eq:dist.X.lev} we obtain
	\begin{align*}
		\Ebb_j\bigg[\sup_{0<s\leq 1}|X_s|^\kappa\bigg]
		\leq  \Ebb_j\Bigg[\Bigg(\sum_{j\in S} \sup_{0<s\leq 1}|X_s^{(j)}| + \sum_{n=1}^{N_1}\sum_{\substack{i,\ell\in S\\ i\neq \ell,\, q_{i\ell}\neq 0}} |Z^{i\ell}_{X,n}| \Bigg)^\kappa\Bigg]. 
	\end{align*}
	\normal
	Via Minkowski's inequality and \eqref{jensenestimate} this yields
	\begin{align*}
		\lefteqn{\EE_j\big[\sup_{0<s\leq 1}|X_1|^\kappa\big]^{1/\kappa}}\\ &\leq  \EE_j\Bigg[ \Bigg(\sum_{j\in S} \sup_{0<s\leq 1}|X_s^{(j)}| \Bigg)^\kappa \Bigg]^\frac{1}{\kappa}+\EE_j\Bigg[\Bigg(\sum_{n=1}^{N_1}\sum_{\substack{i,\ell\in S\\ i\neq \ell,\, q_{i\ell}\neq 0}} |Z^{i \ell}_{X,n}| \Bigg)^\kappa \Bigg]^{1/\kappa}\\
		&\leq |S|^\frac{\kappa-1}{\kappa}\Bigg(\sum_{j\in S} \EE\left[\sup_{0<s\leq t}|X_s^{(j)}|^\kappa\right]\Bigg)^\frac{1}{\kappa}+\Big(|S|^2-|S|\Big)^\frac{\kappa-1}{\kappa} \EE_j\big[N_1^\kappa\big]^\frac{1}{\kappa} \Bigg( \sum_{\substack{i,\ell\in S\\ i\neq \ell,\, q_{i\ell}\neq 0}} \EE_j\Big[|Z^{i\ell}_{X,1}|^\kappa\Big] \Bigg)^\frac{1}{\kappa}
		<\infty,
	\end{align*}
	where, for the second summand, we used the independence of $N_1$ and $Z^{i\ell}_{X,n}$, that $Z^{i\ell}_{X,n}\sim Z^{i\ell}_{X,1}$, and that $N_1$ has finite moments of all orders.
\end{proof}

\begin{proof}[Proof of Lemma \ref{LemmaIntegrabilitySupremumMAP}]
	W.l.o.g. we consider $\sup_{s\in[0,1]}$. The proof for general $t<\infty$ is similar.\\
	As earlier on, let $N^{i\to j}_t$ denote the number of transitions of $J$ from the state $i$ to the state $j$ up to time $t>0$. 
	The L\'evy processes $X^{(j)}$, $j\in S,$ are independent and also independent of the variables $Z^{ij}_{X,\ell}$, and of the counting processes $N^{i\to j}$ for all $i,j\in S$. Thus by monotonicity of  the exponential function  \pa and using \eqref{eq:dist.X.lev}, we get
	\begin{align}
		\EE_\pi\bigg[\sup_{s\in[0,1]} \ee^{ \kappa |X_s|}\bigg] 
		&= \EE_\pi\bigg[\exp\Big(\kappa\cdot\sup_{s\in[0,1]}|X_s|\Big)\bigg]\nonumber  \\
		& \leq
		\EE_\pi\bigg[\exp\bigg(\kappa\bigg(\sum_{j\in S}\sup_{s\in[0,1]} |X^{(j)}_s|+\sum_{\substack{i,j\in S\\i\neq j}}\sum_{\ell=1}\p{N_1^{i\to j}} |Z^{ij}_{X,\ell}| \bigg)\bigg)\bigg]\nonumber \\
		&= \bigg( \prod_{j\in S} \EE_\pi \bigg[\sup_{s\in[0,1]} \ee^{\kappa  |X^{(j)}_s| } \bigg] \bigg) \cdot  \EE_\pi\bigg[\exp\bigg(\kappa \sum_{\substack{i,j\in S\\i\neq j}}\sum_{\ell=1}\p{N_1^{i\to j}}|Z^{ij}_{X,\ell}|\bigg)\bigg].\label{eq-momentexistencehelp}
	\end{align} \normal 
	The first factor does not involve the initial distribution of $J$, and by \cite[Thms. 25.18 and 25.3]{SATO_LPinfinitelydivisibledistributions} for all $j\in S$
	$$\EE\Big[\sup_{s\in[0,1]} \ee^{\kappa  |X^{(j)}_s| } \Big]<\infty \quad \Leftrightarrow \quad \EE\Big[ \ee^{\kappa  |X^{(j)}_1| } \Big]<\infty \quad \Leftrightarrow \quad \int_{|x|\geq 1}\ee^{\kappa |x|}\nu_{X^{(j)}}(\rmd x)<\infty.$$
	It thus remains to prove that the second factor in \eqref{eq-momentexistencehelp} is finite if and only if $\EE\left[\ee^{\kappa |Z^{ij}_{X,1}|}\right]<\infty$ for all $i,j\in S$.	Since  $(N_1^{i\to j})_{i,j\in S}$ is independent of the i.i.d.\ sequences $\{Z\p{ij}_{X,\ell}, \ell\in\NN\}$,
	conditioning on $(N^{i\to j}_1)_{i,j\in S}=(n_{ij})_{i,j\in S}=: \mathbf{n}$, yields
	\[\begin{split}
		\EE_\pi\bigg[\exp\bigg(\kappa \sum_{\substack{i,j\in S\\i\neq j}}\sum_{\ell=1}\p{N_t^{i\to j}}|Z^{ij}_{X,\ell}|\bigg)\bigg]&
		= \sum_{\mathbf{n}\in \NN^{|S|\times |S|}} \PP_\pi \Big((N^{i\to j}_1)_{i,j\in S}= \mathbf{n}\Big) \prod_{\substack{i,j\in S\\i\neq j}}	\Big( \EE_\pi\Big[\ee^{\kappa |Z^{ij}_{X,1}|}\Big] \Big)^{n_{ij}},
	\end{split}
	\]
	which is finite if and only if $\EE\big[\ee^{\kappa |Z^{ij}_{X,1}|}\big]<\infty$ for all $(i,j)\in S^2$ such that $q_{ij}>0$.
\end{proof}

\begin{proof}[Proof of Proposition \ref{detPsiEEexplessthan1}]
	Recall first that the transition probabilities of the embedded discrete time Markov chain of $J$ are given by  $p_{ij}=\frac{q_{ij}}{|q_{ii}|}$ for any $i,j\in S$. Further, the corresponding Lévy processes $X^{(j)}$ and the additional jumps $Z^{ij}$ are all independent, and independent of the holding times of $J$ in any state. Moreover, for an exponentially distributed random variable $T$ with rate $q$, independent of a Lévy process $(Y_t)_{t\geq 0}$, the Laplace transform of $Y_T$, whenever it exists, is given by, 
	\begin{align*}
		\EE[\ee^{\kappa Y_T}]=\frac{q}{q-\psi_Y(\kappa)},
	\end{align*}
	cf. \cite[Eq. (I.3.9)]{STEUTEL},	where $\psi_Y$ is the Laplace exponent of $Y$ and supposed to exist.\\
	Now we return to our setting. W.l.o.g. we consider $\EE_1[\ee^{\kappa X_{\taure_1(1)}}]$, and define the matrices 
	\begin{align*}
		\fat{R}\in\RR^{|S|\times |S|} \; \text{with}\; 	R_{ij}:=&
		\begin{cases}
			0,&i=j,\\
			\frac{q_{ij}}{|q_{ii}|-\psi_{X^{(i)}}(\kappa)}  \EE[\exp(\kappa Z^{ij}_{X,1})],& i\neq j,
		\end{cases}
		\\ \text{and}
		\quad \fat{L}\in\RR^{|S|\times |S|} \; \text{with}\;
		L_{ij}:=&
		\begin{cases}
			0,&i=1\text{ or }j=1,\\
			R_{ij},& \text{else}.
		\end{cases}
	\end{align*}
	Clearly, by our assumptions, $\fat R$ (and hence $\fat L$) is well-defined for the chosen value of $\kappa$. \\
	Using the notations as introduced in Section \ref{S22} we note that 
	$T_{N_{\taure_1(1)}}=\taure_1(1)$ and hence, using the defined matrices, we have
	\begin{align}
		\EE_1\Big[\ee^{\kappa X_{\taure_1(1)}}\Big]
		&= \EE_1\Big[\ee^{\kappa(X_{\taure_1(1)}-X_{T_{N_{\taure_1(1)}-1}})}\cdots  \ee^{\kappa(X_{T_2}-X_{T_1})} \ee^{\kappa X_{T_1}}\Big]  \nonumber \\
		&= \sum_{n\geq 2}\EE_1\left[\ee^{\kappa(X_{T_n}-X_{T_{n-1}})}\cdots  \ee^{\kappa(X_{T_2}-X_{T_1})} \ee^{\kappa X_{T_1}}\right] \PP(N_{\taure_1(1)} =n) \nonumber \\
		& =\fat e_1^\top\cdot\fat R^\top\cdot \bigg(\sum_{k\geq 0}(\fat L^{\top})^k\bigg)\fat R^\top\cdot\fat e_1.\label{exponentialreturntimeformula}
	\end{align}
	We thus observe that $\EE_1[\ee^{\kappa X_{\taure_1(1)}}]<\infty$ if and only if the geometric series in \eqref{exponentialreturntimeformula} converges. By \cite[Prop. 9.4.13 and Cor. 9.4.10]{Bernstein} this holds in particular if $\| \fat L \|_{\rm row}= \| \fat L \|_{\infty,\infty}<1.$ This in turn is equivalent to \eqref{eq-condmomentreturn} by the definition of $\fat L$, thus finishing the proof of $(i)$. \\
	For $(ii)$ assume that $\EE_1[\ee^{\kappa X_{\taure_1(1)}}]<\infty$ such that the series in \eqref{exponentialreturntimeformula} converges to
	\begin{align*}
		\sum_{k\geq 0}(\fat L^{\top})^k=(\fat I-\fat L^\top)^{-1}.
	\end{align*}
	Recall that for any square matrix $\fat A\in\RR^{d\times d}$ the adjugate matrix $\fat A^\dagger=(A^\dagger_{ij})_{i,j=1\ldots, d}$ is defined via
	\begin{align*}
		A^\dagger_{ij}:=(-1)^{i+j}\det(\fat A_{(j,i)}),
	\end{align*}
	where $\fat A_{(i,j)}\in \RR^{(d-1)\times(d-1)}$ is the matrix that results from deleting the $i$th row and $j$th column of $\fat A$. Since for an invertible matrix $\fat A$ it holds $\fat A^{-1}= (\det \fat A)^{-1} \fat A^\dagger$, and since $R_{11}=0$, \eqref{exponentialreturntimeformula} yields
	\begin{align}
		\EE_1\Big[\ee^{\kappa X_{\taure_1(1)}}\Big]
		&=
		\fat e_1^\top \cdot \fat R^\top (\fat I-\fat L^\top)^{-1}\cdot \fat R^\top\cdot\fat e_1\notag
			=\frac{1}{\det(\fat I-\fat L^\top)} 
		\fat e_1^\top \cdot \fat R^\top (\fat I-\fat L^\top)^{\dagger}\cdot \fat R^\top\cdot\fat e_1\notag
		\\
		&=
		\frac{1}{\det(\fat I-\fat L^\top)}\sum_{i=2}^{|S|}\sum_{j=2}^{|S|} (-1)^{i+j}R_{i1}\det\Big((\fat I-\fat L^\top)_{(j,i)}\Big)R_{1j}\notag
		\\
		&=
		\frac{1}{\det(\fat I-\fat L)}\sum_{i=2}^{|S|}\sum_{j=2}^{|S|} (-1)^{i+j}R_{i1}\det\Big((\fat I-\fat L)_{(i,j)}\Big)R_{1j}. \label{exponentialreturndoublesum}
	\end{align}
	Further, a Laplace expansion of  $\det(\fat I-\fat R)$ leads to
	\begin{align}
		\det(\fat I-\fat R)
		&=
		\sum_{i=1}^{|S|} (-1)^{1+i}(\fat I-\fat R)_{i1}\det\Big((\fat I-\fat R)_{(i,1)}\Big)\notag
		\\
		&= (1-R_{11})\det((\fat I-\fat R)_{(1,1)}) + \sum_{i=2}^{|S|} (-1)^{1+i}(-R_{i1})\det\Big((\fat I-\fat R)_{(i,1)}\Big)\notag\\
		&= \det(\fat I-\fat L) +  \sum_{i=2}^{|S|} (-1)^{1+i}(-R_{i1})\sum_{j=2}^{|S|} (-1)^j (-R_{1j}) \det\Big(\big((\fat I-\fat R)_{(i,1)}\big)_{(1,j-1)}\Big) \nonumber \\
		&= \det(\fat I-\fat L) - \sum_{i=2}^{|S|} \sum_{j=2}^{|S|} (-1)^{i+j} R_{i1} R_{1j}\det\Big((\fat I-\fat L)_{(i,j)}\Big). \label{exponentialreturndet(I-R)expansion}
	\end{align}
	Thus, combining \eqref{exponentialreturndoublesum} and \eqref{exponentialreturndet(I-R)expansion} we obtain 
	\begin{align}\label{exponentialreturnfinalequation}
		\EE_1[\ee^{\kappa X_{\taure_1(1)}}]=\frac{1}{\det(\fat I-\fat L)}\left(\det(\fat I-\fat L)-\det(\fat I-\fat R)\right)=1-\frac{\det(\fat I-\fat R)}{\det(\fat I-\fat L)}.
	\end{align}
	To show the claim, it remains to verify that if $\lambda_{\max}^X(\kappa)<0$, then the ratio on the right hand side of \eqref{exponentialreturnfinalequation} is always positive. To this end we rewrite both determinants in terms of $\det\fat\Psi_X(\kappa)$ which gives via \eqref{eq:def.psi.lap.tran} 
	\begin{align}
		\det(\fat I-\fat R)=(-1)^{|S|} \prod_{i=1}^{|S|} \frac{1}{|q_{ii}|-\psi_{X^{(i)}}(\kappa)}\det\fat\Psi_X(\kappa)\label{exponentialreturndeterminantspositive1}
		\\
		\text{and} \quad \det(\fat I-\fat L)=(-1)^{{|S|}-1} \prod_{i=2}^{|S|} \frac{1}{|q_{ii}|-\psi_{X^{(i)}}(\kappa)}\det(\fat\Psi_X(\kappa))_{(1,1)}.\label{exponentialreturndeterminantspositive2}
	\end{align}
	As $\fat\Psi_X(\kappa)$ is assumed to be negative definite, by Silvester's criterion \pa(see e.g. \cite[Thm. 7.2.5]{HornJohnson}) \normal it holds $(-1)^{|S|} \det\fat\Psi_X(\kappa)>0$, while  $(-1)^{|S|-1}\det(\fat\Psi_X(\kappa))_{(1,1)}>0$. Thus \eqref{exponentialreturndeterminantspositive1} and \eqref{exponentialreturndeterminantspositive2} are positive and the statement follows from \eqref{exponentialreturnfinalequation}.
\end{proof}

\begin{proof}[Proof of Lemma \ref{Theoremn_Mean_Stopped_LP_integral}]
	Let $(\gamma,\sigma\p2,\nu)$, $W\p\sigma$, and $\mu$ denote the characteristic triplet, the Gaussian part, and the jump measure of the L\'evy process $X$, respectively. Furthermore, we denote by $\ol\mu(\rmd t,\rmd x):=\mu(\rmd t,\rmd x)-\rmd t\nu(\rmd x)$ the compensated jump measure of $X$. We notice that in the given setting $\int_{(0,\cdot]} \int_\Rbb x(\mu(\rmd s,\rmd x)-\nu(\rmd x)\rmd s)$ is a martingale and, for $\kappa \geq2$, it is a square integrable martingale. \\	
	Because of the L\'evy-It\^o decomposition of $X$, since $\Ebb[|X_1|\p\kappa]<\infty$ and $\kappa\geq1$, we can write
	\[
	X_t=\widetilde\gamma t+W_t\p\sigma+\int_{(0,t]} \int_\Rbb x\,\ol\mu(\rmd s,\rmd x),\quad t\geq0,
	\]
	where $\widetilde\gamma:=\gamma+\int_{\{|x|>1\}}x\,\nu(\rmd x)$. Since $H_-$ is a locally bounded predictable process, this implies
	\[
	\int_{(0,\cdot]} H_{s-}\rmd X_s=\widetilde\gamma\int_{(0,\cdot]}  H_{s-}\rmd s+\int_{(0,\cdot]}  H_{s-}\rmd W\p\sigma_s+\int_{(0,\cdot]} \int_\Rbb H_{s-}x\,\ol\mu(\rmd s,\rmd x).
	\]
	So, by Minkowski's inequality, we deduce
	\begin{align}
		\EE\bigg[\sup_{0\leq t\leq \tau}\bigg|\int_{(0,t]} H_{s-}\rmd X_s\bigg|^\kappa\bigg]^{\frac{1}{\kappa}}
		&=\EE\bigg[\sup_{t\geq0}\bigg|\int_{(0,t\wedge\tau]}H_{s-}\rmd X_s\bigg|^\kappa\bigg]^{\frac{1}{\kappa}}\nonumber   \\
		&\leq |\wt\gamma|\EE\bigg[\sup_{t\geq0}\bigg|\int_{(0,t\wedge\tau]} H_{s-}\rmd s\bigg|^\kappa\bigg]\p{\frac{1}{\kappa}} +\EE\bigg[\sup_{t\geq0}\bigg|\int_{(0,t\wedge\tau]}H_{s-}\rmd W\p\sigma_s\bigg|\p\kappa\bigg]\p{1/\kappa} \nonumber 
		\\ & \quad +\Ebb\bigg[\sup_{t\geq0}\bigg|\int_{(0,t\wedge\tau]}\int_\Rbb H_{s-}x\,\ol\mu(\rmd s,\rmd x)\bigg|^\kappa\bigg]\p{\frac{1}{\kappa}},\label{eq:est.mink.int.drift}
	\end{align}
	Concerning the first summand on the right-hand side in \eqref{eq:est.mink.int.drift}, we have
	\[
	\EE\bigg[\sup_{t\geq0}\bigg|\int_{(0,t\wedge\tau]} H_{s-}\rmd s\bigg|^\kappa\bigg]\leq\Ebb\bigg[\tau^\kappa \sup_{0\leq t\leq \tau}|H_t|^\kappa \bigg]<\infty,
	\]
	where the last estimate follows by \eqref{fun.est.mom} with $\eta=\kappa$, and $\varepsilon>0$ if $\tau$ is not bounded.\\
	For the  Brownian integral in \eqref{eq:est.mink.int.drift}, by the Burkholder-Davis-Gundy inequality, cf. \cite[Cor. IV.(4.2)]{RevuzYor}, we get for some constant \pa $C_\kappa\in(0,\infty)$ \normal
	\begin{equation*}
		\EE\bigg[\sup_{t\geq0}\bigg|\int_{(0,t\wedge\tau]}H_{s-}\rmd W\p\sigma_s\bigg|\p\kappa\bigg]\leq\sigma\p2  C_\kappa  \EE\bigg[\bigg(\int_{(0,\tau]}H_{s-}\p2\rmd s\bigg)\p{\kappa/2}\bigg] \leq\sigma\p2  C_\kappa  \,\Ebb\Big[\tau^{\kappa/2} \sup_{0\leq t\leq \tau}|H_t|^\kappa \Big]<\infty,
	\end{equation*}
	where the last estimate follows from \eqref{fun.est.mom} with $\eta=\kappa/2$ and $\varepsilon>0$, if $\tau$ is not bounded.\\
	It remains to consider the integral with respect to the compensated jump measure in \eqref{eq:est.mink.int.drift}. By \cite[Thm.\ 1]{MaRoe14} there exist constants \pa $C\p 1_\kappa,C\p 2_\kappa \in(0,\infty)$ \normal such that
	\begin{align*}
		\lefteqn{\Ebb\Bigg[\sup_{t\geq0}\bigg|\int_{(0,t\wedge\tau]}\int_\Rbb H_{s-}x\,\ol\mu(\rmd s,\rmd x)\bigg|^\kappa\Bigg]} 	\\
		&\qquad\leq\begin{cases}
			\displaystyle C\p1_\kappa\,\Ebb\bigg[\int_{(0,\tau]}\int_\Rbb|H_{s-}x|\p\kappa\nu(\rmd x)\rmd s \bigg],& \kappa\in[1,2),\\
			\displaystyle
			C\p2_\kappa\,\bigg(\Ebb\bigg[\bigg(\int_{(0,\tau]}\int_\Rbb |H_{s-}x|\p2\nu(\rmd x) \rmd s\bigg)\p{\kappa/2}\bigg]+\Ebb\bigg[\int_{(0,\tau]} \int_\Rbb |H_{s-}x|\p\kappa\nu(\rmd x) \rmd s \bigg]\bigg),& \kappa\in[2,\infty),
		\end{cases}
		\\&\qquad\leq\begin{cases}
			\displaystyle \wt C\p1_\kappa\Ebb\Big[\tau \sup_{0\leq t\leq\tau}|H_t|\p\kappa\Big],& \kappa\in[1,2),\\
			\displaystyle \wt C\p2_\kappa\Ebb\Big[\tau\p{\kappa/2} \sup_{0\leq t\leq\tau}|H_t|\p\kappa\Big]
			+ C\, \Ebb\Big[\tau \sup_{0\leq t\leq\tau}|H_t|\p\kappa\Big],& \kappa\in[2,\infty),\end{cases}
	\end{align*}
	where $\wt C\p1_\kappa:=C\p1_\kappa\int_\Rbb |x|\p\kappa\nu(\rmd x)$, $\wt C\p2_\kappa:=C\p2_\kappa(\int_\Rbb x\p2\nu(\rmd x))\p{\kappa/2}$ and $C:=\frac{\wt C_\kappa\p 1}{C_\kappa\p 1}\, C_\kappa\p 2 \normal$ are finite by the assumptions on $X$. By \eqref{fun.est.mom} with $\eta=1$ or $\eta=\kappa/2$, and with $\varepsilon>0$ if $\tau$ is not bounded, we see that the right-hand side in the previous estimate is finite. The proof is complete.
\end{proof}

\begin{proof}[Proof of Lemma \ref{lem:com.Psii}]
	Following \cite[Thm. II.1.8]{JS00}, the measure $\nu\p{ij}_Z$ is the $\Fbb$-dual predictable projection of $\mu_Z\p{ij}$ if and only if
	\begin{equation}\label{eq:char.com.Psii}
		\Ebb\bigg[\int_{(0,\infty)}\int_\Rbb W(s,x)\mu_Z\p{ij}(\rmd s,\rmd x)\bigg]=\Ebb\bigg[\int_{(0,\infty)} \int_\Rbb W(s,x)\nu\p{ij}_Z(\rmd s,\rmd x)\bigg]
	\end{equation}
	holds for every $\widetilde \Pscr(\Fbb)=\Pscr(\Fbb)\otimes\Bscr(\Rbb)$-measurable non-negative function $(\om,t,x)\mapsto W(\om,t,x)$. Hereby $\Pscr(\Fbb)$ denotes the $\sig$-algebra of $\Fbb$-predictable sets. \\
	Assume that $W(\om,t,x)=B_t(\om)f(x)$, where $(B_t)_{t\geq 0}, B_t\geq0$, is a bounded predictable process and $f\geq0$ a bounded Borel function, then
	\[
	\Ebb\bigg[\int_{(0,\infty)}\int_\Rbb W(s,x)\mu_Z\p{ij}(\rmd s,\rmd x)\bigg]
	=
	\Ebb\bigg[\sum_{n\geq 1}f(Z\p{ij}_n)\mathds{1}_{\{J_{T_n-}=i, J_{T_n}=j\}}B_{T_n}\bigg].
	\]
	As $B_{T_n}$ is $\Fscr_{T_n-}$-measurable, $B$ being predictable, and $Z\p{ij}_{X,n}$ is independent of $\sig(J_{T_n})\vee\Fscr_{T_n-}$ with $Z\p{ij}_{X,n} \overset{d}= Z\p{ij}_{X,1}$, we obtain
	\begin{align*}
		\Ebb\bigg[\int_{(0,\infty)}\int_\Rbb W(s,x)\mu_Z\p {ij}(\rmd s,\rmd x)\bigg]
		&=
		\Ebb\Big[f(Z\p {ij}_1)\Big]\Ebb\bigg[\int_{(0,\infty)} B_s\rmd N^{i\to j}_s\bigg]\\
		&=
		\Ebb\Big[f(Z\p {ij}_1)\Big]\Ebb\bigg[\int_{(0,\infty)} B_s\rmd\phi\p {ij}_s\bigg],
	\end{align*}
	where in the last identity we used the properties of the dual predictable projection $\phi\p{ij}$ given in \eqref{eq-projectNitoj}. 
	Hence, we obtain \eqref{eq:char.com.Psii} for this special choice of $W$. Since the class of predictable functions $W$ of the chosen form generates $\widetilde \Pscr(\Fbb)$, the monotone class theorem (cf. \cite[Thm. I.1.8]{PROTTER_StochIntandSDE}) now yields \eqref{eq:char.com.Psii} for every non-negative predictable bounded function $W$. Finally, consider an arbitrary predictable function $W\geq0$. Then $W\p n:=W\wedge n$ is bounded and by the previous step \eqref{eq:char.com.Psii} holds for $W\p n$. By monotone convergence this implies \eqref{eq:char.com.Psii} for $W$ and hence the statement. 
\end{proof}

\anita
\subsection{Proofs for the results in Section \ref{S3b}}
\normal 

\begin{proof}[Proof of Theorem \ref{MAPmeanTheorem}]
	To prove \eqref{eq:meanhatfatX} we use the integration by parts formula \eqref{integrationbyparts} for $\hatfat{X}$, where we first show that the $\RR\p{|S|}$-valued local martingale $\int_{(0,\cdot]} X_{s-}\rmd \fat{M}_s$  in  \eqref{integrationbyparts} is an $\RR\p{|S|}$-valued centered martingale with respect to $\PP_j$, for every $j$.  Indeed, we have $\EE_j[\sup_{t\in{[0,T]}}|X_t|]=\EE_j[\sup_{t\geq0}|X_{t\wedge T}|]<\infty$, for every fixed $T>0$, because of Lemma \ref{prop:mom.map.lev.ju}. Thus, the claim follows by Lemma \ref{eq:unif.mart.int.M} and we obtain that  $\EE_j[\int_{(0,t]} X_{s-}\rmd \fat{M}_s]=\mathbf{0}$ with $\mathbf{0}$ denoting the column vector of zeros.\\ 
	So, from \eqref{integrationbyparts}, we obtain 
	$$	\EE_j\big[\fat{\hat{X}}_t\big]=\EE_j\big[\fat{\hat{X}}_0\big]+\fatQ^\top \EE_j\bigg[\int_{(0,t]}\fat{\hat{X}}_{s-}\rmd s  \bigg]+\EE_j\bigg[\int_{(0,t]} \fatLambda_{s-}\rmd X_s\bigg]+\EE_j\big[[\fatLambda,X]_t\big].$$
	Using that $X_0=0$ and $\Delta \fat\Lm\p{j}=\pm1$, and recalling the notation introduced in \eqref{eq:def.cou.re.ex}, this yields
	\begin{align} \label{eq-exphelp1}
		\EE_j\big[\fat{\hat{X}}_t\big]
		&= \fatQ^\top \EE_j\bigg[\int_{(0,t]}\fat{\hat{X}}_{s-}\rmd s  \bigg]+ \EE_j\bigg[\bigg(\int_{(0,t)} \mathds{1}_{\lbrace J_{s-}=i\rbrace}\rmd X_s^{(i)}+ \sum_{\ell=1}^{N^{ex(i)}_t}\Delta X_{\tauex_\ell(i)}\bigg)_{i\in S}\bigg] \nonumber \\
		&\qquad +\EE_j\bigg[\bigg(\sum_{\ell=1}^{N_t^{re(i)}}\Delta X_{\taure_\ell(i)}-\sum_{\ell=1}^{N^{ex(i)}_t}\Delta X_{\tauex_\ell(i)}\bigg)_{i\in S}\bigg] 
		\nonumber \\
		&= \fatQ^\top \EE_j\bigg[\int_{(0,t]} \fat{\hat{X}}_{s-}\rmd s  \bigg]+ \EE_j\bigg[\bigg(\int_{(0,t)} \mathds{1}_{\lbrace J_{s-}=i\rbrace}\rmd X_s^{(i)}\bigg)_{i\in S}\bigg]+\EE_j\bigg[\bigg(\sum_{\ell=1}^{N_t^{re(i)}}\Delta X_{\taure_\ell(i)}\bigg)_{i\in S}\bigg].
	\end{align}
	
	As by assumption $\EE_\pi[|X_1|]<\infty$, Lemma \ref{prop:mom.map.lev.ju} implies that the processes $X^{(i)}$, $i\in S$, are L\'evy processes with $\EE[|X^{(i)}_1|]<\infty$ for every $i\in S$. Thus \cite[Lem. 6.1]{BEHME_DistributionalPropertiesGOULevynoise} implies
	\begin{align}
		\EE_j\bigg[\bigg(\int_{(0,t)} \mathds{1}_{\lbrace J_{s-}=i\rbrace}\rmd X_s^{(i)}\bigg)_{i\in S}\bigg]&=\bigg(\EE\big[X^{(i)}_1\big]\int_{(0,t)} \EE_j\big[\mathds{1}_{\lbrace J_s=i\rbrace}\big]\rmd s\bigg)_{i\in S} \nonumber \\
		&=\diag \left(\EE\big[X_1^{(i)}\big]\right)\bigg(\int_{(0,t)} \PP_j(J_s=i)\rmd s\bigg)_{i\in S} \nonumber \\
		&=\diag \left(\EE\left[X_1^{(i)}\right]\right)\int_0^t\ee^{\fatQ^\top s}\rmd s \cdot\fat{e}_j, \label{eq-exphelp2}
	\end{align}
	via \eqref{eq-expfatLambda}. \\
	It remains to consider the last summand in \eqref{eq-exphelp1}. 
	Let $N^{k\to i}_t$ denote the number of transitions of $J$ from the state $k$ to the state $i$ up to time $t>0$. Then rewriting the sum gives
	\begin{align*}
		\EE_j\bigg[\bigg(\sum_{\ell=1}^{N_t^{re(i)}}\Delta X_{\taure_\ell(i)}\bigg)_{i\in S}\bigg]=\bigg(\sum_{k\in S}\EE_j\bigg[\sum_{\ell=1}^{N^{k\to i}_t}Z_{X,\ell}^{ki}\bigg]\bigg)_{i\in S}=\bigg(\sum_{k\in S}\EE_j\Big[N^{k\to i}_t\Big]\EE\Big[Z^{ki}_{X,1}\Big]\bigg)_{i\in S},
	\end{align*}
	by Wald's equality, since $\{Z^{ki}_{X,\ell}, \ell \in\NN\}$ is an i.i.d. sequence independent of $(N_t^{k\to i})_{t\geq 0}$. As e.g. by \cite[Thm. 2]{Yashin_ExpectedNumberOfTransitions}, we have $\EE_j[N^{k\to i}_t]=\int_0^t q_{ki}\PP_j(J_s=k)\rmd s$, we further get
	\begin{align*}
		\bigg(\sum_{k\in S}\EE_j\Big[N^{k\to i}_t\Big]\EE\Big[Z^{ki}_{X,1}\Big]\bigg)_{i\in S}&=\left(\sum_{k\in S} q_{ki}\EE\Big[Z^{ki}_{X,1}\Big]\int_0^t \fat{e}_k^\top \ee^{ \fatQ^\top s }\cdot \fat{e}_j \rmd s \right)_{i\in S}\\
		&=\left(\fatQ\circ \left(\EE\left[Z^{ik}_{X,1}\right]\right)_{i,k\in S}\right)^\top\int_0^t \ee^{\fatQ^\top s}\rmd s\cdot\fat{e}_j.
	\end{align*}
	Inserting this and  \eqref{eq-exphelp2} in \eqref{eq-exphelp1}, and applying  a Fubini argument 
	we obtain the following inhomogeneous ODE for $\EE_j[\hatfat{X}_t]$ as function of $t$ 
	\begin{align*}
		\EE_j\big[\hatfat{X}_t\big]&=\fatQ^\top\int_0^t\EE_j\big[\hatfat{X}_s\big]\rmd s+\fateps[X]\int_0^t\ee^{\fatQ^\top s}\rmd s \cdot \fat e_j.
	\end{align*}
	This is solved uniquely by
	\begin{align*}
		\EE_j \big[\hatfat{X}_t \big]=\ee^{\fatQ^\top t}\EE_j \big[\hatfat{X}_0 \big]+\int_0^t\ee^{\fatQ^\top(t-s)}\fateps[X]\ee^{\fatQ^\top s}\fat e_j\rmd s.
	\end{align*}
	By our standard assumption $X_0=0$ a.s. the first term vanishes and we obtain \eqref{eq:meanhatfatX}. Moreover, multiplying \eqref{eq:meanhatfatX} with $\fatone^\top$ from the left yields \eqref{MAPmeanTheoremintegralformula}, since $\fatone^\top\ee^{\fatQ^\top r}=\fatone^\top$.\\
	To finish the proof of the lemma, note that since $\EE_\pi[|X_t|]<\infty$, we have $\EE_j[|\fat{X}^\#_t|]<\infty$ for all $j\in S$. We may thus compute
	\begin{align*}
		\EE\Big[\fat{X}^\#_{s+t}\Big|\CF_s\Big]
		&=
		\EE\bigg[\fat{\hat{X}}_{t+s}-\fat{\hat{X}}_s+\fat{\hat{X}}_s-\fateps[X]\int_{(0,s]} \fatLambda_r\rmd r-\fateps[X]\int_{(s,s+t]}\fatLambda_r\rmd r\bigg|\CF_s\bigg]\\
		&=
		\EE\bigg[\fat{\hat{X}}_{s+t}-\fat{\hat{X}}_s-\fateps[X]\int_{(s,s+t]} \fatLambda_r\rmd r \bigg|\CF_s\bigg]+\hatfat{X}_s-\fateps[X]\int_{(0,s]} \fatLambda_r\rmd r \\
		&=
		\EE_{J_s}\bigg[\fat{\hat{X}}_t-\fateps[X]\int_{(0,t]} \fatLambda_r\rmd r\bigg] + \fat{X}^\#_s\\
		&=
		\fat{X}^\#_s,
	\end{align*}
	where in the {\pa second to last step} we have used the MAP property \eqref{MAPdefinition} in an extended version as shown e.g. in \cite[Lem. 2.2]{BEHME+SIDERIS_MMGOU}.
\end{proof}

\begin{proof}[Proof of Lemma \ref{lem:Mean_MAPintegral}] 
	Recall that the process $\fat{X}^\#:=\fat{\hat{X}}-\fateps[X]\int_{(0,\cdot]} \fatLambda_s \rmd s$ defined in Theorem \ref{MAPmeanTheorem} is a martingale. Hence, $X^\#:=\mathbf{1}\p\top \fat{X}\p\#$ is a martingale, too, and the stochastic integral $\int_{(0,t]} H_{s-}\rmd X^\#_s$ is a local martingale.  Moreover, it is a true centered martingale, since by the {\pa  triangle} inequality, by the definition of $X^\#$, and by Theorem \ref{Theo_mean_MAP_integral}, we have
	\begin{align*}
		\EE_j\bigg[\sup_{0<u\leq t}\bigg|\int_{(0,u]} H_{s-}\rmd X^\#_s\bigg|\bigg]
		&	\leq\EE_j\bigg[\sup_{0<u\leq t}\bigg|\int_{(0,u]} H_{s-}\rmd X_s\bigg|\bigg]
		+ |\fatone^\top\fateps[X]\fatone| \EE_j\Big[\sup_{0<u\leq t}|H_u|\Big]t
		<\infty.
	\end{align*}
	Thus, 
	\[
	\begin{split}
		\EE_j\bigg[\int_{(0,t]}H_{s-}\rmd X_s\bigg]& 	
		=\EE_j\bigg[\int_{(0,t]}H_{s-}\rmd X_s^\#\bigg] + 	\EE_j\bigg[\int_{(0,t]}H_{s-}\rmd \Big(\fatone^\top \fateps[X] \int_{(0,s]} \fatLambda_{v-} \rmd v\Big)  \bigg] 
		\\&
		=  \fatone^\top \fateps[X]	\EE_j\bigg[\int_{(0,t]}\hatfat{H}_{s-}  \rmd s \bigg],
	\end{split}
	\]
	which yields \eqref{eq:Theorem_Mean_Time_MAPintegral} by a Fubini argument. \\
	To prove \eqref{eq:Theorem_Mean_Time_MAPintegral_vectorvalued}, for $t\geq 0$ and any $i\in S$, observe that with the notation  introduced in  \eqref{eq:def.cou.re.ex}, we get
	$$[X,\fatLambda]_t = \Bigg( \sum_{\ell=1}^{N^{re}_t(i)} \Delta X_{\taure_\ell(i)} - \sum_{\ell=1}^{N^{ex}_t(i)} \Delta X_{\tauex_\ell(i)} \Bigg)_{i\in S},$$
	as all components of $\fatLambda$ only have jumps of size $\pm 1$. Therefore, by the same arguments as in \eqref{eq-exphelp1},
	\begin{align}
		\lefteqn{\int_{(0,t]}\hatfat H_{s-}\rmd X_s+ \int_{(0,t]}H_{s-}\rmd [X,\fatLambda]_s} \nonumber \\
		&= \bigg( \int_{(0,t]} H_{s-} \mathds{1}_{\{J_{s-}=i\}} \rmd X_s^{(i)}\bigg)_{i\in S} + \bigg(\sum_{\ell=1}^{N^{re}_t(i)} H_{\taure_\ell(i) -} \Delta X_{\taure_\ell(i)}  \bigg)_{i\in S}. \label{proofProphelper1}
	\end{align}
	Taking the expectation of the components in the first summand in \eqref{proofProphelper1},
	\cite[Lem. 6.1]{BEHME_DistributionalPropertiesGOULevynoise} gives
	\begin{align*}
		\EE_j\bigg[\int_{(0,t]}H_{s-}\mathds{1}_{\lbrace J_{s-}=i\rbrace}\rmd X^{(i)}_s\bigg]=\EE[X_1^{(i)}] \int_{(0,t]}\EE_j[H_{s-} \mathds{1}_{\{J_{s-}=i\}}]\rmd s,
	\end{align*}
	for all $i,j\in S$, which implies 
	\begin{align}\label{eq:TheoremMeanIntegralwrtMAPpartI}
		\EE_j\bigg[\bigg(\int_{(0,t]}H_{s-}\mathds{1}_{\lbrace J_{s-}=i\rbrace}\rmd X^{(i)}_s\bigg)_{i\in S}\bigg]=\diag\Big[\EE[X_1^{(i)}]\Big]\int_{(0,t]}\EE_j[\hatfat H_s]\rmd s.
	\end{align}
	The second summand in \eqref{proofProphelper1} can be rewritten as
	$$\bigg(\sum_{\ell=1}^{N^{re}_t(i)} H_{\taure_\ell(i) -} \Delta X_{\taure_\ell(i)}  \bigg)_{i\in S} = \bigg( \int_{(0,t]} H_{s-} \rmd Y_s^{(i)}  \bigg)_{i\in S}$$
	for the MAPs $(Y^{(i)},J)$, $i\in S$, defined via $Y^{(i)}_t = \sum_{\ell=1}^{N^{re}_t(i)} \Delta X_{\taure_\ell(i)}$, $t\geq 0$. These have the expectation matrices
	$$\fateps[Y^{(i)}] = \fatQ^\top \circ (\EE[Z_{Y^{(i)},1}^{k,j}])_{k,j \in S}^\top = \fatQ^\top \circ \begin{pmatrix}
		0 &\cdots &0& (\EE[Z_{X,1}^{k,i}])_{k\in S} &0& \cdots& 0 \end{pmatrix}^\top,$$
	{\pa as follows readily} from \eqref{MAPpathdescription}. Thus we conclude via \eqref{eq:Theorem_Mean_Time_MAPintegral} 
	\begin{align}
		\EE_j\bigg[	\bigg(\sum_{\ell=1}^{N^{re}_t(i)} & H_{\taure_\ell(i) -} \Delta X_{\taure_\ell(i)}  \bigg)_{i\in S} \bigg] = \bigg( \EE_j\bigg[\int_{(0,t]} H_{s-} \rmd Y_s^{(i)}  \bigg]\bigg)_{i\in S}
		\nonumber\\
		&= \bigg( \fatone^\top \fateps[Y^{(i)}] \int_{(0,t]} \EE_j\big[\hatfat{H}_s\big] \rmd s \bigg)_{i\in S}= \fatQ^\top \circ \Big( \EE[Z_{X,1}^{k,i}]  \Big)_{k,i\in S}^\top \int_{(0,t]} \EE_j\big[\hatfat{H}_s\big] \rmd s. \label{eq:TheoremMeanIntegralwrtMAPpartII}
	\end{align}
	Combining \eqref{eq:TheoremMeanIntegralwrtMAPpartI} and \eqref{eq:TheoremMeanIntegralwrtMAPpartII} we get the claim.
\end{proof}

\begin{proof}[Proof of Theorem \ref{Thm-MAPvariance}]
	Integration by parts yields $X^2_t=2\int_{(0,t]} X_{s-}\rmd X_s+[X,X]_t$, $t\geq 0$. Thus, applying \eqref{integrationbyparts} we observe that
	\begin{align*}
		X_t^2 \fatLambda_t &= \fatQ^\top \int_{(0,t]} X_{s-}^2 \fatLambda_{s-} \rmd s + \int_{(0,t]} X_{s-}^2 \rmd \fat{M}_s + 2\int_{(0,t]} \fatLambda_{s-} X_{s-} \rmd X_s \\
		& \quad + \int_{(0,t]} \fatLambda_{s-} \rmd[X,X]_s + [X^2, \fatLambda]_t.	
	\end{align*}
	Taking expectations and applying Lemmas \ref{prop:mom.map.lev.ju} and \ref{eq:unif.mart.int.M} this implies
	\begin{align*}
		\lefteqn{\EE_j[X_t^2 \fatLambda_t] }\\&=  \fatQ^\top \int_{(0,t]} \EE_j[X_{s}^2 \fatLambda_{s}] \rmd s + 2\EE_j\bigg[\int_{(0,t]} \hatfat{X}_{s-}\rmd X_s\bigg] + \EE_j\bigg[\int_{(0,t]} \fatLambda_{s-} \rmd[X,X]_s\bigg] + \EE_j\Big[[X^2, \fatLambda]_t\Big].	
	\end{align*}
	Recall from Section \ref{S1} that $([X,X]_t, J)_{t\geq 0}$ is again an $\FF$-MAP and therefore by {\pa Theorem \ref{MAPmeanTheorem}} 
	\begin{align*}
		\EE_j\big[[X,X]_t\big] &= \fatone^\top \fateps\big[[X,X] \big] \int_0^t\ee^{\fatQ^\top s}\rmd s\cdot  \fat e_j.
	\end{align*} Thus via \eqref{eq:Theorem_Mean_Time_MAPintegral_vectorvalued} we obtain
	\begin{align*}
		\EE_j[X_t^2 \fatLambda_t] &=  \fatQ^\top \int_{(0,t]} \pa \EE_j \normal [X_{s}^2 \fatLambda_{s}] \rmd s + 2\bigg( \fateps[X] \int_{(0,t]} \EE_j[\hatfat{X}_s] \rmd s - \EE_j\bigg[\int_{(0,t]} X_{s-}\rmd [X,\fatLambda]_s\bigg]\bigg) \\
		&\quad  + \fateps\big[[X,X]\big]\int_{(0,t]} \EE_j[\fatLambda_{s}]  \rmd s - \EE_j\Big[ \big[[X,X], \fatLambda\big]_t \Big] + \EE_j\Big[[X^2, \fatLambda]_t\Big]\\
		&= \fatQ^\top \int_{(0,t]} \pa \EE_j \normal [X_{s}^2 \fatLambda_{s}] \rmd s + 2  \fateps[X] \int_{(0,t]} \EE_j[\hatfat{X}_s] \rmd s  + \fateps\big[[X,X]\big]\int_{(0,t]} \EE_j[\fatLambda_{s}]  \rmd s.
	\end{align*}
	Inserting \eqref{eq-expfatLambda} and \eqref{eq:meanhatfatX} this yields a first order ODE for $ \EE_j[X_t^2 \fatLambda_t]$ which is uniquely solved by $\EE_j[X_t^2 \fatLambda_t]$ as stated in \eqref{eq:2ndmomentMAP}. \\ 
	Moreover, 
	\begin{align*}
		\fateps\big[[X,X] \big] &= 	\diag\left(\EE\left[ [X^{(j)},X^{(j)}]_1\right]\right)+\fatQ^\top\circ\left(\EE \left[(Z^{ij})^2\right]\right)^\top_{i,j\in S}\\
		&= \diag\left(\var (X^{(j)}_1)\right)+\fatQ^\top\circ\left(\EE \left[(Z^{ij})^2\right]\right)^\top_{i,j\in S},
	\end{align*}
	as stated. Finally, the stated formula for the variance is easily derived from \eqref{eq:2ndmomentMAP} via
	$$\var_j(X_t) = \fatone^\top \EE_j[X_t^2 \fatLambda_t]- \EE_j[X_t]^2.\qedhere$$
\end{proof}

\end{document}